\documentclass{amsart}
\pdfoutput=1
\usepackage{amsmath,amssymb,amsthm,mathrsfs}
\usepackage{enumerate}
\usepackage[colorlinks=true,pagebackref,bookmarks,bookmarksdepth=4]{hyperref}
\usepackage{verbatim}
\usepackage[all]{xy}
\usepackage[T1]{fontenc}

\hypersetup{
  pdfinfo={
     Title={On Fitting ideals of logarithmic vector fields and Saito's criterion},
    Author={Brian Pike},
  }
}

\newtheorem{theorem}{Theorem}[section]
\newtheorem*{theorem*}{Theorem}
\newtheorem{corollary}[theorem]{Corollary}
\newtheorem*{corollary*}{Corollary}
\newtheorem{proposition}[theorem]{Proposition}
\newtheorem*{proposition*}{Proposition}
\newtheorem{lemma}[theorem]{Lemma}

\theoremstyle{definition}
\newtheorem{definition}[theorem]{Definition}
\theoremstyle{definition}
\newtheorem{example}[theorem]{Example}
\newtheorem{remark}[theorem]{Remark}

\numberwithin{equation}{section}

\newcommand{\C}{{\mathbb C}}

\newcommand{\g}{{\mathfrak{g}}}

\newcommand{\p}{{\mathfrak{p}}}

\newcommand{\GL}{{\mathrm{GL}}}
\newcommand{\SL}{{\mathrm{SL}}}

\newcommand{\Hom}{{\mathrm{Hom}}}
\newcommand{\End}{{\mathrm{End}}}
\newcommand{\Sing}{{\mathrm{Sing}}}

\newcommand{\Ocnp}{{\mathscr{O}_{\C^n,p}}}

\newcommand{\sing}{{\mathrm{Sing}}}
\newcommand{\smooth}{{\mathrm{Smooth}}}
\newcommand{\ann}{{\mathrm{ann}}}

\newcommand{\Derlog}[1]{{\mathrm{Der}(-\log #1)}}
\newcommand{\Der}{\mathrm{Der}}

\newcommand{\rank}{{\mathrm{rank}}}
\newcommand{\tr}{{\mathrm{tr}}}
\newcommand{\adj}{{\mathrm{adj}}}
\newcommand{\im}{{\mathrm{im}}}

\newcommand{\Gm}{{\mathbb{G}_\mathrm{m}}}

\newcommand{\bracket}[1]{\left<#1\right>}

\title[Fitting ideals of logarithmic vector fields]{On Fitting ideals of logarithmic vector fields and Saito's criterion}
\date{\today}
\author{Brian Pike}
\address{Dept.\ of Computer and Math\-ematical Sciences,
University  of Tor\-onto at Scarborough, 
1265 Military Trail, 
Toronto, ON M1C 1A4,
Canada}
\email{bapike@gmail.com}



\begin{document}
\begin{abstract}
The germ of an analytic set $(X,p)$ in $\C^n$
has an associated 
$\mathscr{O}_{\C^n,p}$--module $\Derlog{X}$ of
\emph{logarithmic vector fields}, the ambient germs of holomorphic
vector fields tangent to the smooth locus
of $X$.
For a module $L\subseteq \Derlog{X}$ let
$I_k(L)$ be the ideal generated by
the $k\times k$ minors of a matrix of generators for $L$;
these are the Fitting ideals of $\Der_{\C^n,p}/L$.
We aim to:
(i)~find sufficient conditions on $\{I_k(L)\}$ to prove $L=\Derlog{X}$;
(ii)~identify $\{I_k(\Derlog{X})\}$, to provide a necessary condition for
equality;
and
(iii)~provide a geometric interpretation of these ideals.

Even for $(X,p)$ smooth, an example shows that Fitting ideals alone are insufficient
to prove equality, although 
we give a different criterion.
Using (ii) and (iii) in the smooth case, we give partial
answers to (ii) and (iii) for arbitrary $(X,p)$.
When $(X,p)$ is a hypersurface, we give
sufficient
algebraic or geometric
conditions for the
reflexive hull of $L$ to equal $\Derlog{X}$;
for $L$ reflexive, this answers (i) and generalizes
criteria of Saito for free divisors
and Brion for linear free divisors.
\end{abstract}
\maketitle
\tableofcontents

\section*{Introduction}
In \cite{Sa}, Kyoji Saito introduced the notion of a \emph{free divisor}, a
complex hypersurface germ $(X,p)\subset (\C^n,p)$ for which the
associated
$\mathscr{O}_{\C^n,p}$--module $\Derlog{X}_p$ of \emph{logarithmic vector fields} is a free
module, necessarily of rank $n$;
geometrically, these are the vector fields tangent to $(X,p)$.
Although many classes of free divisors have been found,
free divisors remain somewhat mysterious; for instance, it is
not completely understood which hyperplane arrangements are free divisors.

To determine when a set of $n$ elements of
$\Derlog{X}_p$ forms a generating set,
Saito proved a criterion (see Corollary
\ref{cor:saitoscriterion}) 
in terms of the determinant of a presentation matrix
of these elements being
reduced in $\Ocnp$.
Saito's criterion can only be satisfied when $(X,p)$ is a free divisor.
Several questions arise:
\begin{enumerate}
\item \label{en:beyondfree} 
For arbitrary $(X,p)$, is there a necessary and sufficient
condition on a submodule $L\subseteq \Derlog{X}_p$ to have equality?
\item \label{en:saitogeo} What does Saito's criterion mean geometrically?
\end{enumerate}

If $X$ is not empty, $\C^n$, or a free divisor, then
$\Derlog{X}_p$ requires $>n$ generators.
For \eqref{en:beyondfree}, 
it is natural to 
mimic Saito's criterion by considering
some condition on 
the ideals $I_k(L)$, 
$1\leq k\leq n$, 
where $I_k(L)$ is
generated by the $k\times k$ minors of a
matrix of a generating set of $L$;
these are the Fitting ideals of $\Der_{\C^n,p}/L$,
but we shall abuse terminology and call them the \emph{Fitting ideals of
$L$}.
An understanding of the geometric content of these ideals
may also answer \eqref{en:saitogeo}.

%

To this end, in this article
we study these ideals of logarithmic vector fields.
The Lie algebra structure of logarithmic vector fields has been
studied,
in particular in \cite{hausermuller},
but these ideals have not.
The only results we are aware of
(\cite[Proposition 2.2]{hausermuller})
address the radicals
of $I_k(\Derlog{X}_p)$ for $k\geq\dim(X)$.

In \S\ref{sec:notation}, we define these Fitting ideals,
and demonstrate in 
Example \ref{ex:counterexample}
that for non-hypersurfaces,
there is no general condition on the Fitting ideals of $L$
sufficient to prove $L=\Derlog{X}_p$, even when $(X,p)$ is smooth.
Nevertheless, an understanding of 
$\{I_k(\Derlog{X}_p)\}$ 
can provide necessary conditions
for $L$ to equal $\Derlog{X}_p$.
The radicals of these ideals also contain 
information on
various stratifications of $(X,p)$.

In \S\ref{sec:logsmooth}
we study smooth germs.
Theorem \ref{thm:smoothgencriteria} answers \eqref{en:beyondfree}
completely for smooth $(X,p)$, although our criterion is not in terms
of Fitting ideals.
In Propositions \ref{prop:smoothideals1} and \ref{prop:smoothideals2}, we
describe  the
corresponding Fitting ideals and their geometric content.

In \S\ref{sec:logarb} we study arbitrary analytic germs.
Using our understanding of the smooth points of $X$ and a generalization of
the Nagata--Zariski Theorem due to \cite{eisenbud-hochster}
(see Theorem \ref{thm:eh}),
we give upper bounds for $I_k(\Derlog{X}_p)$ in Theorem
\ref{thm:onminors} and Remark \ref{rem:improveminors}.
(For instance, it follows that
if $J$ is a prime ideal defining an irreducible component of
$(X,p)$, then a certain
Fitting ideal is contained in $J$, but not $J^2$; see
Remark \ref{rem:reduced}.)
We describe the geometric meaning of certain Fitting ideals
in Proposition 
\ref{prop:componentreduced}.

In \S\ref{sec:duals}, we follow Saito and
recall a useful duality between modules of vector fields
and modules of meromorphic $1$-forms.
With this, under mild conditions
the double dual of a module $L$ of vector fields can be identified
with a larger module of vector fields that we call the
\emph{reflexive hull} of $L$.
For a hypersurface $(X,p)$, $\Derlog{X}_p$ is reflexive and hence
equals its reflexive hull.

In \S\ref{sec:hypersurfaces},
we apply our earlier work to a hypersurface $(X,p)$.
Proposition \ref{prop:hypersurface} addresses the geometric content of
$I_n(\Derlog{X}_p)$, and hence answers \eqref{en:saitogeo}.
Theorem \ref{thm:generalizesaito}
gives sufficient (algebraic or geometric) conditions for
the reflexive hull of a module $L$ of logarithmic
vector fields 
to equal $\Derlog{X}_p$.
When $L$ is already reflexive,
as is the case for $L$ free,
this gives a condition for
$L=\Derlog{X}_p$.
In particular, this recovers Saito's criterion for free divisors
(Corollary \ref{cor:saitoscriterion}),
and a criterion of Michel Brion (Corollary \ref{cor:brions})
for linear free divisors
associated to representations of linear algebraic groups.
It was an attempt to
understand and generalize Brion's result that 
originally motivated this work.

We are grateful 
to Ragnar-Olaf Buchweitz
and Eleonore Faber
for many helpful conversations.

\section{Fitting ideals}
\label{sec:notation}
We begin with some notation that we maintain for the whole
paper.

\subsection{Notation}
For an open subset $U\subseteq \C^n$,
let $\mathscr{O}_U$ denote the sheaf of holomorphic functions on
$U$ and
let $\Der_{U}$ denote the sheaf of holomorphic vector fields on
$U$.
For a sheaf $\mathscr{S}$ (respectively, $\mathscr{S}_{U}$), let
$\mathscr{S}_p$ (resp., $\mathscr{S}_{U,p}$) denote the stalk at
$p\in U$.
For an ideal $I$, let $V(I)$ denote the common zero set of $f\in I$.
For a set or germ $S$ and point $p\in U$, let $I(S)$ be the ideal of functions vanishing on $S$,
and let $\mathscr{M}_p$ be the maximal ideal of functions vanishing at
$p$;
the ambient ring is given by context. 
For a germ $g$,
a specific choice of representative will always be denoted by
$g'$ or $g''$.

All of our analytic sets and analytic germs will be reduced.
Except for \S\ref{subsec:lfds}, by abuse of terminology a \emph{Zariski closed}
(respectively, \emph{Zariski open})
subset shall mean an
analytic subset (resp., the complement of an analytic subset).

For an analytic set $X\subseteq U$
the
coherent sheaf $\Derlog{X}$ of
\emph{logarithmic vector fields of $X$}
is defined by
\begin{equation}
\label{eqn:derlog}
\Derlog{X}(V)=\{\eta\in \Der_U(V): \eta(I(X))\subseteq I(X)\},
\end{equation}
where $V$ is an arbitrary open subset of $U$ and
$I(X)\subseteq \mathscr{O}_U(V)$ is the ideal of functions on $V$ vanishing
on $X$.
Each $\Derlog{X}(V)$ is an $\mathscr{O}_U(V)$--module closed under the
Lie bracket of vector fields,
consisting of those vector fields on $V$ tangent to
(the smooth points of) $X$.
For an analytic germ $(X,p)$, 
$\Derlog{X}_p$ is defined as $\Derlog{X'}_p$ for any representative
$X'$ of $(X,p)$; it is an $\mathscr{O}_{\C^n,p}$--module closed under
the Lie bracket.
See \cite{Sa,hausermuller} for 
an introduction to logarithmic vector fields.

If $f:M\to N$ is a holomorphic map between complex manifolds, let
$df_{(p)}:T_p M\to T_{f(p)} N$ denote the derivative at $p\in M$. 
For a complex vector space $L$ of vector fields defined at $p$,
define
a subspace of $T_p\C^n$ by
\begin{equation}
\bracket{L}_p=\{\eta(p): \eta\in L\}.
\end{equation}

\subsection{Fitting ideals}

Let $U$ be an open subset of $\C^n$ and let $L\subseteq \Der_U(U)$ be a
$\mathscr{O}_U(U)$--module
of holomorphic vector fields on $U$.
Let $(z_1,\dots,z_n)$ be holomorphic coordinates on $U$.
Choose a set of generators
$\eta_1,\ldots,\eta_m$ of $L$.
Writing
$\eta_j=\sum_{i=1}^n a_{ij}\frac{\partial}{\partial z_i}$,
form a $n\times m$ matrix $A=(a_{ij})$ with entries in
$\mathscr{O}_U(U)$.
We call $A$ a \emph{Saito matrix} for $L$.
For any $1\leq k\leq n$,
consider the $\mathscr{O}_U(U)$ ideal
$I_k(L)$
generated by $k\times k$ minors of a Saito matrix of $L$.

To see that $I_k(L)$ is well-defined, we 
give an equivalent definition.
Identifying $\Der_U(U)$ with $\mathscr{O}_U(U)\{\frac{\partial}{\partial
z_i}\}_{i=1..n}$ gives an exact sequence
\begin{equation*}
\xymatrix{
\mathscr{O}_U^m
\ar[r]^-A &
\Der_U(U) \ar[r] &
 \Der_U(U)/L \ar[r] & 
0.}
\end{equation*}
Thus, $I_k(L)$ is also the $(n-k)$th Fitting ideal
of $\Der_U(U)/L$, and hence well-defined.
By abuse of terminology, we shall refer to
$\{I_k(L)\}$ as the \emph{Fitting ideals of $L$}.

For a submodule $L\subseteq \Der_{\C^n,p}$, 
the ideals
$I_k(L)\subseteq \mathscr{O}_{\C^n,p}$ may be defined
similarly.
Of particular interest are
the ideals $I_k(\Derlog{X}_p)$, $k=1,\ldots,n$,
when $(X,p)$ is an analytic germ in $\C^n$.
For instance, 
the radicals of
these ideals
encode the
\emph{logarithmic stratification} of $(X,p)$,
a (not necessarily finite)
decomposition of $(X,p)$ into smooth strata
along which the germ is biholomorphically trivial
(see \cite[\S3]{Sa}).
Note that when this stratification is finite, it is equal to the
canonical Whitney stratification of $(X,p)$
(\cite[Prop.~4.5]{DP-matrixsingII}).
Since this stratification is a rather subtle property of $(X,p)$, we do not expect to
be able to describe $I_k(\Derlog{X}_p)$ explicitly.
However, if $L\subseteq \Derlog{X}_p$, then
any property known about $I_k(\Derlog{X}_p)$
gives a necessary condition to have $L=\Derlog{X}_p$.

\subsection{Inadequacy of Fitting ideals}
An easy example shows that even when $(X,p)$ is 
smooth, for $L\subseteq \Derlog{X}_p$ the ideals
$\{I_k(L)\}_k$ cannot detect
whether $L=\Derlog{X}_p$.

\begin{example}
\label{ex:counterexample}
Let $(X,0)$ be the origin in $\C^2$, defined by coordinates $x=y=0$.
The following logarithmic vector fields generate
$\Derlog{X}_p$:
$$
\eta_1=x\frac{\partial}{\partial x},\quad
\eta_2=y\frac{\partial}{\partial x},\quad
\eta_3=x\frac{\partial}{\partial y},\quad
\eta_4=y\frac{\partial}{\partial y}.
$$
Let $L=\mathscr{O}_{\C^2,0}\{\eta_2,\eta_3,\eta_1-\eta_4\}$.
Then $L\subsetneq \Derlog{X}_p$, even though
$I_k(L)=I_k(\Derlog{X}_p)$ for $k=1,2$
and $L$ is closed under the Lie bracket.

Similar examples hold for the origin in higher dimensions, as the
action of $\SL(n,\C)$ and $\GL(n,\C)$ on $\C^n$ have the same
orbit structure,
and a simple argument shows that the two modules of
vector fields generated by these actions have
the same Fitting ideals.
This example may then be extended to apply
to all smooth germs of codimension $>1$.
\end{example}

\section{Logarithmic vector fields of smooth germs}
\label{sec:logsmooth}
We first study a smooth analytic germ $(X,p)$ in $\C^n$.
Because of the coherence of 
the sheaf of logarithmic vector fields and the genericity of
smooth points,
we will later use this to
study non-smooth germs. 

The (only) example is:

\begin{example}
\label{ex:smoothgerm}
Let $(X,p)$ in $\C^n$ be a smooth germ of dimension $d<n$.
Choose local coordinates $y_1,\ldots,y_n$ for $\C^n$ near $p$ so that
the ideal of germs vanishing on $X$ is
$I(X)=(y_1,\ldots,y_{n-d})$ in $\mathscr{O}_{\C^n,p}$.
Then the vector fields
\begin{equation}
\label{eqn:vfssmoothgerms}
\left\{\frac{\partial}{\partial y_{k}}\right\}_{k=n-d+1,\ldots,n}
\qquad
\mathrm{and}
\qquad
\left\{y_j\frac{\partial}{\partial y_i}\right\}_{i,j\in\{1,\ldots,n-d\}}
\end{equation}
are certainly in $\Derlog{X}_p$.
\end{example}

We begin by describing a criterion, Theorem
\ref{thm:smoothgencriteria},
for having a complete set of
generators for $\Derlog{X}_p$.
It will follow that 
\eqref{eqn:vfssmoothgerms} is a complete set
of generators for Example \ref{ex:smoothgerm}.

\subsection{A derivative for vector fields}
First, we make a definition used in our criterion.
For $\eta\in\Der_{\C^n,p}$ with $\eta(p)=0$, we shall
define a directional derivative
$d(\widehat{\eta})_{(p)}:T_p \C^n\to T_p \C^n$ of $\eta$ at $p$ by
choosing local coordinates $x_1,\ldots,x_n$ near $p$, writing
$\eta=\sum_{i=1}^n g_i \frac{\partial}{\partial x_i}$,
and then
defining
\begin{equation}
d(\widehat{\eta})_{(p)}(v)=\sum_{i=1}^n
d(g_i)_{(p)}(v)\frac{\partial}{\partial x_i}(p),
\end{equation}
where
we use the canonical coordinate on $\C$ to identify $T_0\C$ with $\C$
so that $d(g_i)_{(p)}(v)\in\C$
(or, interpret $d(g_i)_{(p)}(v)$ as a directional
derivative).
This operation does not depend on a choice of coordinates, as shown by
applying the 
following lemma to $\{x_i-x_i(p)\}_{i}$:

%
\begin{lemma}
\label{lemma:whatiswidehat}
If $\eta(p)=0$ and $f(p)=0$, then
$d(\eta(f))_{(p)}= df_{(p)}\circ d(\widehat{\eta})_{(p)}$.
\end{lemma}
\begin{proof}
Choose coordinates $(x_1,\ldots,x_n)$ near $p$.
Since both sides are linear in $\eta$, it suffices to
prove the claim when 
$\eta=g_i \frac{\partial}{\partial x_i}$,
with $g_i(p)=0$.
In this case, 
\begin{align*}
d(\eta(f))_{(p)}(v)
&=d(g_i)_{(p)}(v)\cdot \frac{\partial f}{\partial x_i}(p)+g_i(p)\cdot
d\left(\frac{\partial f}{\partial x_i}\right)_{(p)}(v) \\
&=df_{(p)}\left(d(g_i)_{(p)}(v)\frac{\partial}{\partial x_i}(p)
\right) +0\\
&=df_{(p)}\left( d(\widehat{\eta})_{(p)}(v)\right).\qedhere
\end{align*}
\end{proof}

\subsection{Criterion for a generating set}

We are now able to state our criterion.

\begin{theorem}
\label{thm:smoothgencriteria}
Let
$(X,p)$ be an analytic germ in $\C^n$ which is of pure dimension
$d<n$.
Let $\eta_1,\ldots,\eta_m\in\Derlog{X}_p$, with
$L_{\C}=\C\{\eta_i\}_{i=1,\ldots,m}$ and
$L=\mathscr{O}_{\C^n,p}\{\eta_i\}_{i=1,\ldots,m}$.
Let $L_{\C,0}$ (respectively, $L_{0}$) denote the set of $\xi\in L_{\C}$
(respectively, $\xi\in L$) with $\xi(p)=0$.
In $\mathscr{O}_{\C^n,p}$, let
$I=I(X)$ be the ideal of functions
vanishing on $X$, and let $\mathscr{M}_p$ be the maximal ideal.
The following are equivalent:
\begin{enumerate}
\item
\label{en:a}
$(X,p)$ is smooth and $L=\Derlog{X}_p$;
\item
\label{en:b}
$\dim(\bracket{L}_p)=d$ and the map
$$\alpha:L_{\C,0}\to\End_{\C}(I/\mathscr{M}_p\cdot I)$$
defined by
$\alpha(\xi)=(f\mapsto \xi(f))$
is surjective;
\item
\label{en:c}
$\dim(\bracket{L}_p)=d$ and the map
$$\beta:L_{\C,0}\to\End_{\C}( T_p \C^n/\bracket{L}_p)$$
defined by
$\beta(\xi)=d(\widehat{\xi})_{(p)}$
is surjective.
\end{enumerate}
In \eqref{en:b} or \eqref{en:c}, we could equivalently have used $L_{0}$ instead of
$L_{\C,0}$.
\end{theorem}

The smoothness of $(X,p)$ will follow from 
this lemma.

\begin{lemma}
\label{lemma:smoothpt}
If $(X,p)$ is an analytic germ in $\C^n$ of dimension $d$,
with $p\in X$,
then $\dim(\bracket{\Derlog{X}_p}_p)\leq d$, with equality 
if and only if $(X,p)$ is smooth.
\end{lemma}

\begin{proof}
Let $k=\dim(\bracket{\Derlog{X}_p}_p)$.
By a theorem of Rossi \cite[Rossi's Theorem, (c)]{hausermuller},
$(X,p)$ has dimension $\geq k$; hence, $k\leq d$.
If $k=d$, then by (b) of Rossi's Theorem, there is a reduced analytic germ
$Y$ such that $(X,p)\cong (\C^d,0)\times Y$; but then $Y$
is of 
dimension $0$, and hence smooth. 
Conversely, if $(X,p)$ is smooth, then 
\cite[Existence Lemma]{hausermuller} shows that $k\geq d$, and thus $k=d$.
\end{proof}

\begin{proof}[Proof of Theorem \ref{thm:smoothgencriteria}]
We first show that $\dim(\bracket{L}_p)=d$ implies that $(X,p)$ is
smooth.  Since $L\subseteq \Derlog{X}_p$,
$\bracket{L}_p\subseteq \bracket{\Derlog{X}_p}_p$.
Since by Lemma \ref{lemma:smoothpt} $\bracket{\Derlog{X}_p}_p$ has dimension
$\leq d$,
we have $\bracket{L}_p=\bracket{\Derlog{X}_p}_p$
and
both spaces have dimension $d$.  By Lemma \ref{lemma:smoothpt} again, $(X,p)$
is smooth.

Thus we may assume in all cases that $(X,p)$ is smooth,
and
$\bracket{L}_p=\bracket{\Derlog{X}_p}_p=T_pX$ have dimension $d$.
Choose
local coordinates $(y_1,\ldots,y_n)$ vanishing at $p$ 
so that
$I=(y_1,\ldots,y_{n-d})$, just as in Example \ref{ex:smoothgerm}.

To show that $\alpha$ is well-defined, 
let $D$ be the submodule of $\Derlog{X}_p$ consisting of vector fields
vanishing at $p$, and let
$\alpha':D\to\End_\C(I/\mathscr{M}_p\cdot I)$ be
defined by $\alpha'(\xi)(f)=\xi(f)$.
Let $\xi\in D$.
Since $\xi\in\Derlog{X}_p$, by definition 
the application of $\xi$ to functions gives a linear map $I\to I$.
If $f\in I$ and $g\in
\mathscr{M}_p$, then
$\xi(g)\in\mathscr{M}_p$ (as $\xi(p)=0$) and $\xi(f)\in I$.  By the
product rule, $\xi(f\cdot g)\in \mathscr{M}_p\cdot I$ and hence
$\xi(\mathscr{M}_p\cdot I)\subseteq \mathscr{M}_p\cdot I$.
Thus the map $f\mapsto \xi(f)$ descends to an endomorphism of 
$I/\mathscr{M}_p\cdot I$, as claimed, and
so $\alpha'$ is well-defined.
Since $L_{\C,0}\subseteq L_{0}\subseteq D$,
restricting $\alpha'$ to one of these subspaces gives a well-defined
map, including $\alpha=\alpha'|_{L_{\C,0}}$.

We will show that \eqref{en:a} implies \eqref{en:b}.
The vector space $I/\mathscr{M}_p I$ is generated by
$\{y_1,\ldots,y_{n-d}\}$, and (as in Example \ref{ex:smoothgerm})
$y_j\frac{\partial}{\partial y_i}\in L_0$ has
\begin{equation*}
\left(y_j \frac{\partial}{\partial y_i}\right)(y_k)=\delta_{ik} y_j.
\end{equation*}
Since 
it follows that
$\left\{\alpha'\left(y_j\frac{\partial}{\partial y_i}\right)\right\}_{1\leq i,j\leq
n-d}$ is a $\C$-basis of $\End_{\C}(I/\mathscr{M}_{p}\cdot I)$,
$\alpha'|_{L_{0}}$ is surjective.

However, we must show that $\alpha=\alpha'|_{L_{\C,0}}$ is surjective.
If $\xi\in L$ and $g\in\mathscr{M}_p$, then $g\cdot \xi\in L_0$ and 
$g\cdot \xi(I)\subseteq \mathscr{M}_p\cdot I$; hence, 
$\mathscr{M}_p\cdot L\subseteq \ker(\alpha')$.
Since $\dim(\bracket{L}_p)=d$, we may choose a basis for $L_{\C}$
consisting of $\eta_1,\ldots,\eta_d$ (with $\{\eta_i(p)\}$ linearly
independent) and elements of $L_{\C,0}$.
Let $\zeta\in L_0$, and write
$\zeta=\sum_{i} g_i \eta_i +\sum_{j} h_j \xi_j$, where $\xi_j\in
L_{\C,0}$.  Evaluating at $p$ and using the linear independence of
$\{\eta_i(p)\}$, we see that all $g_i\in\mathscr{M}_p$.
Let $\zeta'=\sum_{j} h_j(p) \xi_j\in L_{\C,0}$, and observe that
$$\zeta-\zeta'=\sum_{i} g_i \eta_i+\sum_j (h_j-h_j(p))\xi_j\in
\mathscr{M}_p\cdot L,$$
and hence
$\alpha'(\zeta)=\alpha'(\zeta')=\alpha(\zeta')$.
Thus $\alpha$ is surjective as well,
giving us \eqref{en:b}.

To show that \eqref{en:b} implies \eqref{en:a},
we will first show
\begin{equation}
\label{eqn:keralphaprime}
\ker(\alpha')\subseteq \mathscr{M}_p\cdot N,
\end{equation}
where $N$ is the submodule of $\Derlog{X}_p$ generated by the vector
fields in \eqref{eqn:vfssmoothgerms}.
If $\xi\in\ker(\alpha')$, then $\xi\in D$, and hence $\xi(p)=0$ and 
$\xi=\sum_{i} a_i\frac{\partial}{\partial y_i}$ with each
$a_i\in\mathscr{M}_p$.
Since $\xi\in\ker(\alpha')$, 
$\xi(y_i)\in \mathscr{M}_p\cdot I$ for $i=1,\ldots,n-d$.
Hence, we may write
\begin{equation}
\label{eqn:xiequals}
\xi=\sum_{i=1}^{n-d} \left(\sum_{j=1}^{n-d} b_{ij} y_j\right)
\frac{\partial}{\partial y_i}+\sum_{n-d<i\leq n}
a_i\frac{\partial}{\partial y_i},
\end{equation}
where $b_{i,j},a_i\in\mathscr{M}_p$.
Examining \eqref{eqn:xiequals}, it is clear that
$\xi\in\mathscr{M}_p\cdot N$, proving \eqref{eqn:keralphaprime}.

Now let $\zeta\in\Derlog{X}_p$.
Since $\{\eta_i(p)\}_{i=1}^d$ is a basis for
$\bracket{L}_p=\bracket{\Derlog{X}}_p$,
write $\zeta(p)=\sum_{i=1}^d \lambda_i \eta_i(p)$ for $\lambda_i\in\C$.
Let $\zeta'=\zeta-\sum_{i=1}^d \lambda_i \eta_i$, so that
$\zeta'(p)=0$.
By the surjectivity of $\alpha$, there exists a $\xi\in L_{\C,0}$ such
that $\alpha'(\zeta')=\alpha(\xi)$.
Thus,
$$
\zeta-\left(\sum_{i=1}^d \lambda_i \eta_i+\xi\right)\in\ker(\alpha'),
$$
and so by \eqref{eqn:keralphaprime},
$\zeta\in L+\mathscr{M}_p\cdot N$.
Consequently,
$$
\Derlog{X}_p\subseteq L+\mathscr{M}_p\cdot N\subseteq
L+\mathscr{M}_p\cdot \Derlog{X}_p\subseteq \Derlog{X}_p,$$
so that by Nakayama's Lemma, $L=\Derlog{X}_p$.
This proves \eqref{en:b}.

To show that $\beta$ in \eqref{en:c} is well-defined,
observe that by the smoothness of $(X,p)$,
$\bracket{L}_p=T_p X=\cap_{f\in I} \ker(df_{(p)})$.
Let $\xi\in\Derlog{X}_p$ vanish at $p$.
Since for any $f\in I$ we have
$\xi(f)\in I$, by Lemma \ref{lemma:whatiswidehat},
$d(\widehat{\xi})_{(p)}(T_p X)\subseteq \ker(df_{(p)})$.
Since this is true of all $f\in I$,
we have
$d(\widehat{\xi})_{(p)}(T_p X)\subseteq T_p X$,
and so
$d(\widehat{\xi})_{(p)}:T_p \C^n\to T_p \C^n$ factors through the quotient
$T_p \C^n\to T_p\C^n/T_pX$ to give a well-defined element
$\beta(\xi)$.

We now show the equivalence of \eqref{en:b} and \eqref{en:c}.
Define the $\C$-linear map
$\overline{\rho}:I\to \Hom_{\C} (T_p \C^n/T_p X,T_0 \C)$
by $\overline{\rho}(f)=df_{(p)}$; this is well-defined since
for $f\in I$, $T_p X\subseteq \ker(df_{(p)})$.
Since $y_1,\ldots,y_{n-d}\in I$, we know $\overline{\rho}$ is
surjective.
If $g\in\mathscr{M}_p$ and $f\in I$, then 
$d(f\cdot g)_{(p)}=0$, and hence $\mathscr{M}_p\cdot I\subseteq
\ker(\overline{\rho})$.
Conversely, if $f=\sum_{i=1}^{n-d} a_i y_i\in I$ is in $\ker(\overline{\rho})$, then
$0=df_{(p)}=\sum_{i=1}^{n-d} a_i(p) d(y_i)_{(p)}$; by the linear
independence of $d(y_i)_{(p)}$, all $a_i\in\mathscr{M}_p$.
Hence $\mathscr{M}_{p}\cdot I=\ker(\overline{\rho})$, 
and so $\overline{\rho}$ factors through to a vector space isomorphism
$\rho: I/\mathscr{M}_p\cdot I\to \Hom_{\C}(T_p\C^n/T_p X,T_0\C)$.

For any $\xi\in\Derlog{X}_p$ vanishing at $p$,
Lemma \ref{lemma:whatiswidehat} shows that we have
the following commutative diagram.
\begin{equation}
\label{eqn:rhodiagram}
\begin{split}
\xymatrix{
I/\mathscr{M}_p\cdot I
\ar[r]^{\alpha(\xi)} \ar[d]_{\rho}
&
I/\mathscr{M}_p\cdot I
\ar[d]_{\rho} \\
\Hom_{\C}(T_p\C^n/T_p X,T_0\C)
\ar[r]^{(\beta(\xi))^*}
&
\Hom_{\C}(T_p\C^n/T_p X,T_0\C)
}
\end{split}
\end{equation}
Since $\rho$ is an isomorphism, 
it follows from
\eqref{eqn:rhodiagram}
that $\alpha$ is surjective if and only if $\beta$
is surjective, and hence \eqref{en:b} is equivalent to \eqref{en:c}.
\end{proof}

\begin{remark}
Since $\bracket{L}_p=T_p X$, the map $\beta$ in Theorem 
\ref{thm:smoothgencriteria}
describes how vector fields behave with respect to the normal space
to $X$ at $p$ in $\C^n$, $T_p\C^n/T_p X$.
So, too, does $\alpha$:
algebraically, the quotient of the Zariski tangent spaces
of $\C^n$ and $X$ at $p$
is the dual of
$$(I+\mathscr{M}_p^2)/\mathscr{M}_p^2\cong I/(I\cap
\mathscr{M}_p^2)= I/\mathscr{M}_p\cdot I,$$
where the equality is because in this case,
$I\cap \mathscr{M}_p^2=\mathscr{M}_p\cdot I$.
The map $\rho$ constructed in the proof of Theorem
\ref{thm:smoothgencriteria} is essentially an identification between
the dual of this ``Zariski normal space'' to $X$ at $p$ and
the dual of the usual normal space to $X$ at $p$.
\end{remark}

\begin{remark}
If $(X,p)$ is an arbitrary analytic germ
with an irreducible component 
of dimension $d$ in $\C^n$,
then by the coherence of $\Derlog{X}$ and Theorem
\ref{thm:smoothgencriteria},
$\Derlog{X}_p$ requires at least $d+(n-d)^2$ generators;
the component of smallest dimension gives the strongest bound.
Are germs for which $\Derlog{X}_p$ is minimally generated
in some way special, as is the case for hypersurfaces (see
\S\ref{sec:fd})?
\end{remark}

\subsection{Fitting ideals for smooth germs}

We now compute the Fitting ideals associated to $\Derlog{X}_p$ for $(X,p)$
smooth.

\begin{proposition}
\label{prop:smoothideals1}
Let $(X,p)$ be a smooth germ of dimension $d$ in $\C^n$, with $d<n$.
In $\mathscr{O}_{\C^n,p}$, we have 
\begin{equation}
I_k(\Derlog{X}_p)=
\left\{\begin{array}{ll}
 (I(X))^{k-d} & \textrm{if $k>d$} \\
 (1) & \textrm{otherwise}
\end{array}\right..
\end{equation}
\end{proposition}

\begin{proof}
Choose coordinates as in Example \ref{ex:smoothgerm}, with
$I(X)=(y_1,\ldots,y_{n-d})$.  Any Saito matrices have rows
corresponding to the coefficients of 
$\frac{\partial}{\partial y_1},\ldots,\frac{\partial}{\partial y_n}$.
By Theorem \ref{thm:smoothgencriteria}, the vector
fields of \eqref{eqn:vfssmoothgerms} generate $\Derlog{X}_p$.
A Saito matrix $M$ of these generators might be, in block form, 
$$
M=\begin{pmatrix}
Y &   &        &   & 0\\
  & Y & \\
  &   & \ddots & \\
  &   &        & Y & \\
0 &   &        &   & J_{d}
\end{pmatrix},
\textrm{ where }
Y=\begin{pmatrix} y_1 & y_2 & \cdots & y_{n-d} \end{pmatrix},
$$
and for any $\ell$, $J_\ell$ denotes the $\ell\times \ell$ identity matrix.
Then $I_k(\Derlog{X}_p)$ is generated by the
determinants of the $k\times k$ submatrices of
$M$.
If $k\leq d$, then a $k\times k$ submatrix of $J_d$ is
nonsingular, and hence $I_k(\Derlog{X}_p)$ contains $1$.
Now let $k>d$.
Any submatrix which uses more than one column from the same $Y$ block will have
determinant zero, as there is an obvious relation between the columns.
Also, any submatrix which uses a non-symmetric choice of rows and
columns of $J_{d}$ will have a zero row or column.
Hence, the only nonzero generators of $I_k(\Derlog{X}_p)$ will be
$$
\det\begin{pmatrix}
y_{i_1} & \\
 & \ddots & \\
 &     & y_{i_{\ell}} & \\
 &     &              & J_{k-\ell}
\end{pmatrix}=
\prod_{j=1}^\ell y_{i_j},
$$
where each $i_j\in\{1,\ldots,n-d\}$, and (since $0\leq k-\ell\leq d$) $k-d\leq \ell\leq k$.
Thus $I_k(\Derlog{X}_p)$ is generated by all monomials of degree $k-d$
in $\{y_1,\ldots,y_{n-d}\}$, which is exactly as claimed.
\end{proof}

\subsection{Geometry of the Fitting ideals for smooth
germs}

We may give a geometric interpretations to one of these ideals.
Let $\alpha$ and $\beta$ be defined as in Theorem
\ref{thm:smoothgencriteria}. 

\begin{proposition}
\label{prop:smoothideals2}
Let $(X,p)$ be a smooth germ of dimension $d$ in $\C^n$, with $d<n$.
Let $L\subseteq \Derlog{X}_p$ be a submodule, and let
$L_0\subseteq L$ be the submodule of vector fields vanishing at $p$.
Then the following are equivalent:
\begin{enumerate}
\item
\label{en:sm2a}
in $\mathscr{O}_{\C^n,p}$, we have 
$I_{d+1}(L)\nsubseteq \mathscr{M}_p^2$;
\item
\label{en:sm2b}
$\dim(\bracket{L}_p)=d$ and the map $\alpha|_{L_0}$ is nonzero;
\item
\label{en:sm2c}
$\dim(\bracket{L}_p)=d$ and the map $\beta|_{L_0}$ is nonzero.
\end{enumerate}
\end{proposition}

We need a lemma for the proof.
Let $M(p,q,\C)$ be the space of $p\times q$ matrices with complex
entries.
Since $M(p,q,\C)$ is a vector space, there is a canonical
identification $T_{A}M(p,q,\C)\simeq M(p,q,\C)$ for all $A$.

\begin{lemma}
\label{lemma:linearalgebra}
Let $\gamma:(-\epsilon,\epsilon)\to M(m,m,\C)$ be a
differentiable curve. Define $A=\gamma(0)$, $B=\gamma'(0)$,
and $\delta(t)=\det(\gamma(t))$. 
Then the following are equivalent:
\begin{enumerate}[(i)]
\item
\label{en:deltaprime}
$\delta(0)=0$ and $\delta'(0)\neq 0$;
\item
\label{en:bkera}
$A$ has rank $m-1$ and $B\cdot \ker(A)\nsubseteq \im(A)$.
\end{enumerate} 
\end{lemma}
\begin{proof}
Since the statement is obvious when $m=1$, assume $m>1$.
Let $\adj(A)$ denote the adjugate of $A\in M(p,p,\C)$.
By Jacobi's formula,
\begin{equation}
\label{eqn:jacobis}
\delta'(0)=\tr(\adj(\gamma(0))\cdot \gamma'(0))=\tr(\adj(A)\cdot B).
\end{equation}
By hypothesis and \eqref{eqn:jacobis}, in both cases we must have
$\rank(A)=m-1$, so assume this.
Since $\det(A)=0$,
we have $A\cdot \adj(A)=\adj(A)\cdot A=\det(A)\cdot I=0$, and thus
\begin{equation}
\label{eqn:adjAinclusions}
\im(\adj(A))\subseteq \ker(A)\qquad\mathrm{and}\qquad
\im(A)\subseteq \ker(\adj(A)).
\end{equation}
By the first inclusion of
\eqref{eqn:adjAinclusions}, $\rank(\adj(A))\leq 1$;
as also $\adj(A)\neq 0$, we
have $\rank(\adj(A))=1$
and may write $\adj(A)=\lambda\cdot \omega$ for some
nonzero
$\lambda\in M(m,1,\C)$ and $\omega\in M(1,m,\C)$.
By dimensional considerations, the inclusions of
\eqref{eqn:adjAinclusions} are equalities, and we may identify
these spaces using the structure of $\adj(A)$:
\begin{equation}
\label{eqn:adjAequals}
\begin{split}
\im(\adj(A))&=\ker(A)=\C\cdot \lambda \\
  \text{and}
  \qquad \im(A)&=\ker(\adj(A))=\{z\in M(m,1,\C):\omega\cdot z=0\}.
\end{split}
\end{equation}
By \eqref{eqn:jacobis}, 
\begin{equation*}
\label{eqn:longcalculation}
\delta'(0)
=\tr(\lambda\cdot (\omega \cdot B))
 =(\omega \cdot B) \cdot\lambda,
\end{equation*}
and in particular,
$\C\cdot \delta'(0)=\omega\cdot B\cdot \ker(A)$.
The equivalence follows from this equation and \eqref{eqn:adjAequals}.
\end{proof}

\begin{proof}[Proof of \ref{prop:smoothideals2}]
By \eqref{eqn:rhodiagram}, \eqref{en:sm2b} is equivalent to
\eqref{en:sm2c}.
Let $M$ be a Saito matrix of a generating set
of $L$, 
and $N$ be a $(d+1)\times(d+1)$ submatrix of $M$.
By Proposition \ref{prop:smoothideals1}, $\det(N)\in\mathscr{M}_p$.
Observe that if $\dim(\bracket{L}_p)<d$, then $\rank(N(p))<d$ and so
by Lemma \ref{lemma:linearalgebra} we have $\det(N)\in\mathscr{M}_p^2$.

It remains only to show that if $\dim(\bracket{L}_p)=d$, then
\eqref{en:sm2a} is equivalent to \eqref{en:sm2b}.
Assume, then, that 
$\dim(\bracket{L}_p)=d$, and hence $T_pX=\bracket{L}_p$.
Choose coordinates as in Example \ref{ex:smoothgerm}, with
$I(X)=(y_1,\ldots,y_{n-d})$.
Choose $\eta_1,\ldots,\eta_d\in L$
spanning $\bracket{L}_p$ at $p$ such that
$L=\mathscr{O}_{\C^n,p}\{\eta_i\}_{i=1}^d+L_0$.
Let $M$ be a Saito matrix of $L$ with columns corresponding to $\eta_1,\ldots,\eta_d$
and a set of generators of $L_{0}$. 
It is enough to consider determinants of submatrices of $M$.
Observe that $M(p)$ is zero, except for the $d\times d$ submatrix
$A$ of rank $d$ 
corresponding to $\eta_1,\ldots,\eta_d$ and
$\frac{\partial}{\partial
y_{n-d+1}},\ldots,\frac{\partial}{\partial y_n}$.
By Lemma \ref{lemma:linearalgebra}, it is necessary for
any $(d+1)\times(d+1)$ submatrix $N$ of $M$
to contain $A$ in order to have $\det(N)\notin
\mathscr{M}_p^2$. 

Thus, consider a submatrix $N$ of $M$ containing $A$, as well as a
column corresponding to $\xi\in L_{0}$ and a row corresponding to
$\frac{\partial}{\partial y_j}$, with $j<n-d+1$. 
For any $v\in T_p\C^n$, let $\gamma_v(t)=N(p+tv)$, and  
observe that $\gamma_v(0)$ has rank $d$,
$\im(\gamma_v(0))$ corresponds to a projection of $T_p X$
onto $V=\C\{\frac{\partial}{\partial y_j}(p),
\frac{\partial}{\partial
y_{n-d+1}}(p),\ldots,\frac{\partial}{\partial y_n}(p)
\}$,
$\ker(\gamma_v(0))$ corresponds to $\C\cdot \xi$,
and $\gamma'_v(0)\cdot \ker(\alpha_v(0))$ corresponds to a projection of
$\C\cdot d(\widehat{\xi})_{(p)}(v)$ onto $V$.
By Lemma \ref{lemma:linearalgebra},
$(\det\circ \gamma_v)'(0)\neq
0$ if and only if the $\frac{\partial}{\partial y_j}$--component of  
$d(\widehat{\xi})_{(p)}(v)$ is nonzero;
moreover, such a $v\notin T_p X$ since
$d(\widehat{\xi})_{(p)}(T_pX)\subseteq T_p X$.

If $\det(N)\in \mathscr{M}_p\setminus
\mathscr{M}_p^2$, then there exists a $v$ so that 
$\frac{d}{dt}(\det(N(p+tv)))|_{t=0}\neq 0$,
and by the above 
$d(\widehat{\xi})_{(p)}(v)$ has a nonzero $\frac{\partial}{\partial
y_j}$--coefficient, so that $\beta(\xi)\neq 0$. 
Conversely, if $\beta(\xi)\neq 0$, then there exists a $v\notin T_p X$ and a
$j<n-d+1$ so that  
$d(\widehat{\xi})_{(p)}(v)$ has a nonzero $\frac{\partial}{\partial
y_j}$--coefficient, and by the above argument
the corresponding submatrix $N$ has
$\det(N)\notin \mathscr{M}_p^2$. 
\end{proof}

\section{Logarithmic vector fields of arbitrary germs}
\label{sec:logarb}
In this section, we study the logarithmic vector fields for an arbitrary
analytic germ $(X,p)$ in $\C^n$ by using the properties of logarithmic vector
fields at nearby smooth points.
We shall study the associated ideals and their geometric meaning.

\subsection{Choice of representatives}
We shall often
choose representatives of an analytic germ $(X,p)$ and
a collection of vector fields in
$\Derlog{X}_p$,
and then study the behavior of these representatives at points near to $p$.
It is convenient to choose representatives on an open $U$ containing
$p$ so that the behavior of the representatives
on $U$ reflect only the behavior of the original germs.
For any germ $f$, a specific choice of representative
will be denoted by $f'$ or $f''$.

\begin{lemma}
\label{lemma:goodrep}
Let $(X,p)$ be an analytic germ in $\C^n$ and let
$\eta_1,\ldots,\eta_m\in\Derlog{X}_p$.
For any choice of representatives of these germs on open neighborhoods
of $p$, there exists an
open neighborhood $U$ of $p$ such that:
\begin{enumerate}
\item \label{cond:repsdefined} all representatives are defined on $U$;
\item \label{cond:irrdecomp} if $X=\cup_{s\in S} X_s$ is an
irreducible decomposition of germs, then there is a corresponding irreducible
decomposition $X'=\cup_{s\in S} X'_s$ on $U$
where $(X'_s,p)=(X_s,p)$;
\item \label{cond:inderlogprime} each $\eta'_i\in\Derlog{X'}(U)$;
\item \label{cond:thesegenderlog} if
$\mathscr{O}_{\C^n,p}\cdot\{\eta_i\}_{i=1}^m=\Derlog{X}_p$, then
$\{\eta'_i\}_{i=1}^m$ generates the sheaf $\Derlog{X'}$ on $U$;
\item \label{cond:polydisc} $U$ is, e.g., a polydisc, so that $U$ is
convex (using real line segments).
\end{enumerate}
\end{lemma}
\begin{proof}
Without loss of generality, assume that the hypothesis of
\eqref{cond:thesegenderlog} is satisfied (by, e.g., increasing $m$).
Without mentioning it explicitly, each $U_i$ will be an open
neighborhood of $p$.
Let $X''$ and $\eta''_1,\ldots,\eta''_m$ be representatives, and let
$U_1$ be an open set on which
these representatives are all defined.

Decompose $(X,p)$ into irreducible components as in \eqref{cond:irrdecomp},
and find $U_2\subseteq U_1$ containing representatives
$X''_s$ of each $(X_s,p)$.
Since each $(X_s,p)$ is irreducible, there exists a 
$U_3\subseteq U_2$ with the property that
for any open neighborhood $V\subseteq U_3$ containing $p$,
$V\cap X''_s$ is irreducible in $V$,
and hence $\cup_{s\in S} \left( V\cap X''_s\right)$ is an irreducible
decomposition of $V\cap \left(\cup_{s\in S} X''_s\right)$ in $V$. 
Since $\cup_{s\in S} X''_s$ is a representative of $(X,p)$,
it and $X''$ are equal on some $U_4\subseteq U_3$.

By the coherence of $\Derlog{(U_4\cap X'')}$
and the hypothesis of \eqref{cond:thesegenderlog},
there is a $U_5\subseteq U_4$ on which
representatives of $\eta_1,\ldots,\eta_m$ can be found which generate
the sheaf $\Derlog{(U_5\cap X'')}$;
by the definition of germ, there is a $U_6\subseteq U_5$ on which
$\eta''_1,\ldots,\eta''_m$ generate $\Derlog{(U_6\cap X'')}$.

Finally, let $U\subseteq U_6$ be a polydisc containing $p$,
and
set $X'=U\cap X''$, $X'_s=U\cap X''_s$,
and $\eta'_i=\eta''_i|_{U}$.
It is easily checked that all conditions are satisfied.
\end{proof}

Observe that by \eqref{cond:irrdecomp}, $X'_s$ is irreducible in $U$,
and hence $I(X'_s)\subseteq \mathscr{O}_U(U)$ is prime and
$X'_s$ is pure-dimensional.
By \eqref{cond:polydisc}, if $f\in\mathscr{O}_U(U)$ has $f(q)=0$ for
some $q\in U$, then by Hadamard's Lemma we may write
$f=\sum_{i=1}^n f_i \cdot (x_i-x_i(q))$, 
with
$f_i\in\mathscr{O}_U(U)$.
More generally, if $f\in\mathscr{O}_U(U)$ has
$f\in\mathscr{M}_q^k\subseteq \mathscr{O}_{U,q}$, then
$f\in\mathscr{M}_q^k\subseteq \mathscr{O}_U(U)$.

\subsection{Symbolic powers of ideals}

Before proceeding, we recall the concept of the symbolic power of an
ideal.
Let $R$ be a commutative Noetherian ring with identity.
For a prime ideal $\p\subseteq R$, 
it is not necessarily true that $\p^\ell$ is $\p$-primary
(e.g., \cite[\S3.9.1]{eisenbud});
the $\p$-primary component of $\p^\ell$ is called the
\emph{$\ell$th symbolic power of $\p$} and is denoted by
$\p^{(\ell)}$.

A result originally due to Zariski and Nagata 
says that $\p^{(\ell)}$ is the ideal
of functions which vanish to order $\geq \ell$ on $V(\p)$.

\begin{theorem}[{\cite[Corollary 1]{eisenbud-hochster}}]
\label{thm:eh}
Let $\p$ be a prime ideal of a ring $R$, and let $S$
be a set of maximal ideals $\mathscr{M}$ containing $\p$ such that
$R_\mathscr{M}/\p_{\mathscr{M}}$ is a regular local ring, and
$\cap_{\mathscr{M}\in S} \mathscr{M}=\p$.
Then
\begin{equation*}
\p^{(\ell)}
\supseteq \bigcap_{\mathscr{M}\in S} \mathscr{M}^\ell
\supseteq \bigcap_{\substack{\mathscr{M}\supseteq \p \\
\text{$\mathscr{M}$ maximal}}} \mathscr{M}^\ell,
\end{equation*}
with equalities if $R$ is regular.
\end{theorem}

In particular, we have
\begin{corollary}
\label{cor:eisenbudhochster}
If $\p$ is a prime ideal of $\mathscr{O}_U(U)$, and
$N\subseteq \smooth(V(\p))$ is such that the closure
$\overline{N}=V(\p)$, then
\begin{equation*}
\p^{(\ell)}
=\bigcap_{p\in N} \mathscr{M}_p^\ell
=\bigcap_{p\in V(\p)} \mathscr{M}_p^\ell.
\end{equation*}
\end{corollary}
\begin{proof}
$\mathscr{O}_U$ is a regular ring.
Since $\overline{N}=V(\p)$, we have
$$
\cap_{p\in N} \mathscr{M}_p
=I(N)
=I(\overline{N})
=\p.$$
Then apply Theorem \ref{thm:eh}.
\end{proof}

\begin{remark}
Since primary decomposition commutes with localization,
so does the symbolic power.
\end{remark}

It is also useful to know the following characterization of
$\p^{(\ell)}$ when $\p$ defines a complete intersection.

\begin{lemma}
\label{lemma:ci}
If $\p$ is a prime ideal in $\mathscr{O}_{U,p}$ generated by a regular
sequence, then
$\p^{(\ell)}=\p^\ell$ for all $\ell\geq 1$.
\end{lemma}
\begin{proof}
$\mathscr{O}_{U,p}$ is Cohen-Macaulay
so this follows from, e.g.,
Proposition 3.76 of \cite{vasconcelos}.
\end{proof}

\subsection{Fitting ideals}

We can now prove an upper bound for the Fitting ideals of
logarithmic vector fields.

\begin{theorem}
\label{thm:onminors}
Let $(X,p)$ be an analytic germ in $\C^n$, with
$X=\cup_{s\in S} X_s$ the decomposition into irreducible
components. 
Let $1\leq k\leq n$.
Then in $\mathscr{O}_{\C^n,p}$,
we have
\begin{equation}
\label{eqn:eqnonminors}
I_k(\Derlog{X}_p)\subseteq\bigcap_{\substack{s\in S \\ k>\dim(X_s)}}
  (I(X_s))^{(k-\dim(X_s))},
\end{equation}
where the exponents denote symbolic powers and are
\emph{sharp}, that is, changing the RHS of \eqref{eqn:eqnonminors} by either
increasing the exponents or intersecting with any nontrivial symbolic
power of any $I(X_s)$ with $k\leq\dim(X_s)$ would make the statement
false.
Moreover, the difference between the two sides of
\eqref{eqn:eqnonminors} is supported on $\Sing(X)$.
\end{theorem}

For example, $I_n(\Derlog{X}_p)$ reflects both the
irreducible components of $(X,p)$ and the dimensions of these
components.

\begin{proof}
Let $\eta_1,\ldots,\eta_m$ generate $\Derlog{X}_p$.
Find representatives of $(X,p)$ and each $\eta_i$ on an open
neighborhood $U$ of $p$, as in Lemma \ref{lemma:goodrep}.
Let $M=\Derlog{X'}(U)$, and observe that it is generated by 
$\eta'_1,\ldots,\eta'_m$.

For $s\in S$, let $N_s=\smooth(X')\cap X'_s
=\smooth(X'_s)\setminus \left(\cup_{r\in S\setminus\{s\}}
X'_r\right)$.
Since $N_s$ is a dense open subset of $X'_s$,
$\overline{N_s}=X'_s$.

Let $q\in N_s$, and let $k>\dim(X_s)$.
Since $\eta'_1,\ldots,\eta'_m$ generate $\Derlog{X'}_q$ as a
$\mathscr{O}_{U,q}$ module, the localization of $M$ at $q$ is equal to
$\Derlog{X'}_q$.
Observe that localizing $M$ commutes with taking minors of a
presentation matrix of $M$.
Thus by Proposition \ref{prop:smoothideals1},
as $\mathscr{O}_{U,q}$ modules,
\begin{equation}
\label{eqn:Ikm}
\mathscr{O}_{U,q}\cdot I_k(M)
=I_k(\mathscr{O}_{U,q}\cdot M)
=I_k(\Derlog{X'}_q)
=I(X'_s)^{k-\dim(X_s)}.
\end{equation}
Since $q\in X_s$, by \eqref{eqn:Ikm},
$\mathscr{O}_{U,q}\cdot I_k(M)\subseteq \mathscr{M}_q^{k-\dim(X_s)}$.
By a version of Hadamard's lemma for holomorphic functions,
it follows that
in $\mathscr{O}_U$,
$I_k(M)\subseteq \mathscr{M}_q^{k-\dim(X_s)}$.
As this is true for all $q\in N_s$,
by Corollary \ref{cor:eisenbudhochster} it follows that
$I_k(M)\subseteq (I(X'_s))^{(k-\dim(X_s))}$ whenever $k>\dim(X_s)$.

This proves that as $\mathscr{O}_U(U)$ ideals,
\begin{equation}
\label{eqn:eqnonminorsalmost}
I_k(M)\subseteq\bigcap_{\substack{s\in S \\ k>\dim(X_s)}}
  (I(X'_s))^{(k-\dim(X_s))}.
\end{equation}
Since localization commutes with taking minors,
$M$ localized at $p$ is $\Derlog{X}_p$, 
$I(X'_s)$ localized at $p$ is $I(X_s)$,
and as symbolic powers commute with localization,
\eqref{eqn:eqnonminors} follows from \eqref{eqn:eqnonminorsalmost}.
By \eqref{eqn:Ikm} and Lemma \ref{lemma:ci}, the two sides of
\eqref{eqn:eqnonminors} agree at smooth points of $X$.

To show that the exponents are sharp, fix $s\in S$ and $k>\dim(X_s)$.
Suppose that in $\mathscr{O}_{\C^n,p}$,
\begin{equation}
\label{eqn:ikdbad}
I_k(\Derlog{X}_p)\subseteq (I(X_s))^{(k-\dim(X_s)+1)}.
\end{equation}
As above, choose representatives on an open set $U$ containing $p$.
Each side of \eqref{eqn:ikdbad} is the stalk at $p$ of the 
coherent ideal sheaf $\mathscr{J}$ and $\mathscr{K}$ on $U$ generated by
$I_k(M)$ and $(I(X'_s))^{(k-\dim(X_s)+1)}$, respectively.
It thus follows from \eqref{eqn:ikdbad} that there exists some open
$V\subseteq U$ containing $p$ on which
$\mathscr{J}|_V\subseteq \mathscr{K}|_V$.
Let $q\in V\cap \smooth(X')\cap X'_s$, take the stalks at $q$,
and apply Proposition \ref{prop:smoothideals1} and
Lemma \ref{lemma:ci}
to find that in $\mathscr{O}_{\C^n,q}$,
$$\mathscr{J}_q=(I(X'_s))^{k-\dim(X_s)}\subseteq
\mathscr{K}_q=(I(X'_s))^{k-\dim(X_s)+1};$$
this is a contradiction.

If $s\in S$ and $k\leq \dim(X_s)$, then no symbolic power of $I(X_s)$ can
appear on the right side of \eqref{eqn:eqnonminors}, either by the
same argument, or by the existence of certain vector fields
constructed in \cite[Existence Lemma]{hausermuller}.
\end{proof}

\begin{remark}
\label{rem:improveminors}
Theorem \ref{thm:onminors} is improved significantly by the
following observation.
If $\Derlog{X}_p\subseteq \Derlog{Y}_p$ for some $(Y,p)$, then
$I_k(\Derlog{X}_p))\subseteq I_k(\Derlog{Y}_p)$, and hence applying
Theorem \ref{thm:onminors} to $(Y,p)$ can reduce the upper bound
for $I_k(\Derlog{X}_p)$.
In particular, this applies to $Y=\sing(X)$, $Y=\sing(\sing(X))$, etc.
\end{remark}

\begin{remark}
\label{rem:primarydecomp}
The RHS of \eqref{eqn:eqnonminors}
is part of the primary decomposition of the LHS,
containing those primary ideals corresponding to isolated primes.
\end{remark}

\begin{remark}
\label{rem:reduced}
Let $(X_s,p)$ be an irreducible component of $(X,p)$ and let
$d=\dim(X_s)$.
By Theorem \ref{thm:onminors},
$I_{d+1}(\Derlog{X}_p)$
is contained in $I(X_s)$, but not contained in
$I(X_s)^{(2)}$.
\end{remark}

%
%
\begin{example}
Consider the hypersurface $X$ in $\C^4$ defined by
$xw-yz=0$, with 
$\sing(X)$ the origin.
Then $M=\Derlog{X}_0$ is generated by seven vector fields, and by
Theorem \ref{thm:onminors} and Remark \ref{rem:improveminors},
we know
\begin{align*}
I_4(M)&\subseteq (xw-yz)\cap (x,y,z,w)^4, \\
I_3(M)&\subseteq (x,y,z,w)^3, \\
I_2(M)&\subseteq (x,y,z,w)^2, \\
\text{and }I_1(M)&\subseteq (x,y,z,w);
\end{align*}
in fact, a Macaulay2 computation shows that these are all equalities.
\end{example}

%
\begin{example}
Consider the hypersurface $X_3$ in $\C^4$ defined by $xy(x-y)(xz-yw)=0$.
Inductively let $X_i=\sing(X_{i+1})$, so that
$X_0\subseteq X_1\subseteq X_2\subseteq X_3$, with
each $X_i$ a union of irreducible
complete intersections of dimension $i$.
Then $M=\Derlog{X_3}_0$ is generated by 4 vector fields, and
applying
Theorem \ref{thm:onminors} to all $X_i$
computes $I_4(M)$ and $I_1(M)$ exactly, while
an upper bound with the correct radical is produced for $I_3(M)$.
For $I_2(M)$, Theorem \ref{thm:onminors}
does not
detect that
the logarithmic stratification of $X_3$ is not 
finite, as the vector fields of $M$ have rank $\leq 1$ on
$x=y=0$.
This corresponds to $X_3$ not being biholomorphically
trivial along $x=y=0$.
%
\end{example}

%
\begin{example}
Consider the Whitney umbrella $X_2$ in $\C^3$ defined by $x^2-y^2z=0$.
Then $X_1=\sing(X_2)$ is the smooth set $x=y=0$ and
$M=\Derlog{X_2}_0$ is generated by 4 vector fields.
Applying
Theorem \ref{thm:onminors} to $X_1$ and $X_2$
does not compute any $I_k(M)$ exactly.
Although the bounds for $I_2(M)$ and $I_3(M)$ have the correct
radical,
the upper bound for $I_1(M)$ is $(1)$.
If we recognize that each $\eta\in M$ must be tangent
to the origin, then Theorem \ref{thm:onminors} and Remark
\ref{rem:improveminors} compute
$I_1(M)$ and $I_3(M)$
exactly.
\end{example}

\begin{example}
Consider the variety $X$ defined by the ideal $I$ generated by
the $2\times 2$ minors of
a generic $3\times 3$ symmetric matrix.
Then $X$ has dimension $3$, and $\sing(X)$ is the origin.
Here, the symbolic powers of $I$ differ from the 
usual powers of $I$.
The module $M=\Derlog{X}_0$ is generated by $24$ vector fields.
Theorem \ref{thm:onminors} and Remark \ref{rem:improveminors}
compute $I_1(M),\ldots,I_5(M)$ exactly, while
$I_6(M)$ differs from the computed upper bound.
\end{example}

The sharpness described in
Theorem \ref{thm:onminors} provides a necessary condition on a
submodule of logarithmic vector fields to be a complete generating
set.
However, it is far from sufficient.

\begin{example}
\label{ex:dumbvfs}
Let $f\in\mathscr{O}_{\C^n,p}$ define a reduced hypersurface $(X,p)$.
Then the vector fields
\begin{equation}
\left\{
\frac{\partial f}{\partial x_j}\frac{\partial}{\partial x_i}
-
\frac{\partial f}{\partial x_i}\frac{\partial}{\partial x_j}
\right\}_{1\leq i<j\leq n}
\qquad
\mathrm{and}
\qquad
\left\{f\frac{\partial}{\partial x_i}\right\}_{1\leq i\leq n}
\end{equation}
generate a module $L\subseteq \Derlog{X}_p$ of vector fields which
vanish on $\sing(X)$; nevertheless,
at any nearby smooth point $q$ of $X$,
$L=\Derlog{X}_q$ as $\mathscr{O}_{\C^n,q}$--modules (by, e.g.,
Theorem \ref{thm:smoothgencriteria}).
Since the proof of Theorem \ref{thm:onminors} relied entirely on
behavior at
the
smooth points, each $I_k(L)$ should satisfy
\eqref{eqn:eqnonminors} and the sharpness conditions of Theorem
\ref{thm:onminors}, although in general
$L\neq\Derlog{X}_p$.
%
%
\end{example}

\subsection{Geometry of the Fitting ideals}
Let $L\subseteq \Derlog{X}_p$ be a submodule.
We can give a geometric interpretation to a
certain Fitting ideal of $L$
satisfying the sharpness condition of
Theorem \ref{thm:onminors} with respect to a certain component of
$(X,p)$.
%

\begin{proposition}
\label{prop:componentreduced}
Let $(X,p)$ be an analytic germ in $\C^n$, and let $(X_0,p)$ be an
irreducible component of $(X,p)$ of dimension $d$.
Let $L=\mathscr{O}_{\C^n,p}\{\eta_1,\ldots,\eta_m\}\subseteq
\Derlog{X}_p$.
Choose representatives of $(X,p)$, $(X_0,p)$, and each $\eta_i$,
and let $U$ be an open neighborhood of $p$ 
on which
these representatives satisfy the conditions of Lemma
\ref{lemma:goodrep}.
Let $L'=\mathscr{O}_U\{\eta'_1,\ldots,\eta'_m\}$.
Then the following are equivalent:
\begin{enumerate}
 \item
\label{en:componentreduceda}
$I_{d+1}(L)\nsubseteq (I(X_0))^{(2)}$
in $\mathscr{O}_{\C^n,p}$;
 \item
\label{en:componentreducedb}
for every open neighborhood $V\subseteq U$ containing $p$, there
exists a $q\in \smooth(X')\cap X'_0\cap V$ such that at $q$, $L'$ and
$(X',q)=(X'_0,q)$ satisfy the equivalent conditions of
Proposition \ref{prop:smoothideals2};
 \item
\label{en:componentreducedc}
for every open neighborhood $V\subseteq U$ containing $p$, there
exists a Zariski open, dense subset $W$ of $X'_0\cap V$ such that at every
$q\in W$, $L'$ and 
$(X',q)=(X'_0,q)$ satisfy the equivalent conditions of
Proposition \ref{prop:smoothideals2}.
\end{enumerate}
\end{proposition}
\begin{proof}
The set $A=\{q\in U: \mathscr{O}_{U,q}\cdot I_{d+1}(L')\subseteq
\mathscr{M}_q^2$ in $\mathscr{O}_{U,q}\}$
is defined by the vanishing of $f$ and the partials of $f$, for all
$f\in I_{d+1}(L')$, and hence is a closed analytic subset of $U$.
Since $\sing(X')$ and $X_0'$ are closed analytic sets,
$B=(A\cup \sing(X'))\cap X_0'$ is a closed analytic subset of $X_0'$.
Thus $G=X'_0\setminus B$ is a Zariski open subset of $X'_0$,
and is the set of $q\in N=\smooth(X')\cap X'_0$ where $L'$
and $(X',q)=(X'_0,q)$ satisfy one of the equivalent conditions of Proposition
\ref{prop:smoothideals2}.
Note that $G=N\setminus A$.

Since $G$ is a Zariski open subset of the irreducible $X'_0$,
either $\overline{G}=X'_0$ (and $B$ is nowhere dense in $X'_0$)
or $G=\emptyset$ (and $B=X'_0$).

If $G=\emptyset$, then at every $q\in N$,
$q\in A$ and
hence
(by, e.g., Hadamard's Lemma)
$I_{d+1}(L')\subseteq \mathscr{M}_q^2$ in $\mathscr{O}_U(U)$.
Since $\overline{N}=X'_0$, 
Corollary \ref{cor:eisenbudhochster} shows that
$I_{d+1}(L')\subseteq (I(X_0))^{(2)}$, and localizing at $p$ shows
that \eqref{en:componentreduceda} is false.
However, \eqref{en:componentreducedb} and \eqref{en:componentreducedc} are also false.

If $\overline{G}=X_0'$, then $G$ is dense in $N$, and
\eqref{en:componentreducedb} and
\eqref{en:componentreducedc} are true.
Suppose that \eqref{en:componentreduceda} were false.
Then on some open $V\subseteq U$ containing $p$,
$\mathscr{O}_{U}(V)\cdot I_{d+1}(L')\subseteq (I(X'_0))^{(2)}$ in $\mathscr{O}_U(V)$.
Let $q\in G\cap V$.
By Corollary \ref{cor:eisenbudhochster}, 
$I_{d+1}(L')\subseteq \mathscr{M}_q^2$ in $\mathscr{O}_{U,q}$,
and hence $q\in A$.
But since $A\cap G=\emptyset$, this is a contradiction.
\end{proof}

\section{Reflexive modules}
\label{sec:duals}

Before discussing hypersurfaces in \S\ref{sec:hypersurfaces}, we
recall 
some useful background on reflexive modules.
\newcommand{\holovecs}{\mathrm{Der}}
\newcommand{\holofuncs}{\mathscr{O}}
\newcommand{\holoforms}{\Omega^1}
\newcommand{\merofuncs}{\widetilde{\mathscr{O}}}
\newcommand{\meroforms}{\widetilde{\Omega}^1}
\newcommand{\sheafD}{\mathscr{D}}
\newcommand{\sheafE}{\mathscr{E}}
\newcommand{\sheafF}{\mathscr{F}}
\newcommand{\sheafG}{\mathscr{G}}
\newcommand{\sheafL}{\mathscr{L}}
\newcommand{\sheafM}{\mathscr{M}}
\newcommand{\sheafN}{\mathscr{N}}
\newcommand{\sheafhom}{\mathscr{H}\mathrm{om}}

We adopt the following notation for this section.
Let $U$ be an open subset of $\C^n$.
Let $\holofuncs=\mathscr{O}_{\C^n}|_{U}$ (respectively, $\merofuncs$)
be the sheaf of holomorphic functions (resp., meromorphic functions) on $U$.
Let 
$\holoforms=\Omega^1_{U}$
(respectively, $\meroforms=\widetilde{\Omega}^1_{U}$)
be the $\holofuncs$--module
of holomorphic (resp., meromorphic) $1$-forms on $U$,
and let $\holovecs=\Der_U$ be the
$\holofuncs$--module of holomorphic vector fields on $U$.
For a $\holofuncs$--module $\sheafN$,
denote its $\holofuncs$--dual
by $\sheafN^*=\sheafhom_{\holofuncs}(\sheafN,\holofuncs)$;
for an open $V\subseteq U$,
$\sheafN^*(V)=\Hom_{\holofuncs|_V}(\sheafN|_V,\holofuncs|_V)$ is the
$\holofuncs(V)$--module of $\holofuncs|_V$--module morphisms
$\sheafN|_V\to\holofuncs|_V$.

Let $\theta:\holovecs\times \meroforms\to \merofuncs$,
defined by
$\theta(V)(\eta,\omega)=(x\mapsto \omega(\eta(x)))\in\merofuncs(V)$, be
the standard pairing (or ``inner product'') between vector fields and
$1$-forms,
extended to meromorphic forms. 
Just as for pairings between vector spaces, such a
pairing may sometimes be used to identify the $\holofuncs$--dual of  
a submodule $\sheafN\subseteq \holovecs$
with a submodule of $\meroforms$, and vice-versa.
Following \cite[(1.6)]{Sa}, we have
\begin{lemma}
\label{lemma:duals}
Let $f\in\holofuncs(U)$, $f\neq 0$.
\begin{enumerate}
\item
\label{en:D}
Let $\sheafD$ be a $\holofuncs$--submodule of $\holovecs$, with
$f\cdot \holovecs\subseteq \sheafD\subseteq \holovecs$.
Then $\sheafD^*$ is canonically isomorphic to a
$\holofuncs$--submodule $\sheafM\subseteq \meroforms$,
with
$\holoforms\subseteq \sheafM \subseteq \frac{1}{f}\cdot \holoforms$,
by the pairing $\theta|_{\sheafD\times \sheafM}:\sheafD\times \sheafM\to \holofuncs$.
\item
\label{en:M}
Let $\sheafM$ be a $\holofuncs$--submodule of $\meroforms$, with
$\holoforms\subseteq \sheafM\subseteq \frac{1}{f}\cdot \holoforms$.
Then $\sheafM^*$ is canonically isomorphic to a
$\holofuncs$--submodule $\sheafD\subseteq \holovecs$, with
$f\cdot \holovecs\subseteq \sheafD\subseteq \holovecs$,
by the pairing $\theta|_{\sheafD\times \sheafM}:\sheafD\times \sheafM\to \holofuncs$.
\end{enumerate}
The modules and maps are independent of $f$.
\end{lemma}
\begin{proof}
Fix holomorphic coordinates $x_1,\ldots,x_n$ on $U$.
Throughout, $V$ and $W$ will denote arbitrary open subsets with $W\subseteq
V\subseteq U$.

For \eqref{en:D}, define the $\holofuncs$--submodule $\sheafM\subseteq
\meroforms$ by 
$\sheafM(V)=\{\omega\in \meroforms(V): \theta(V)(\sheafD(V),\omega)\subseteq
\holofuncs(V)\}$,
and define the maps
\begin{align*}
\phantom{\text{and}}\quad\sheafD^*\to \sheafM\qquad\qquad
  &\text{on $V$,\quad}\varphi\mapsto
  \frac{1}{f}\sum_{i=1}^n \varphi(V)\left(f\frac{\partial}{\partial
x_i}\right) dx_i\in \sheafM(V),
\\
\text{and}\quad
\sheafM\to \sheafD^*\qquad\quad
   &\text{on $V$,\quad} \omega\mapsto
\left( W\mapsto \left(\eta\mapsto \theta(W)(\eta,\omega)\right)\right)
\in
\sheafD^*(V).
\end{align*}
Check that the images lie in
the claimed spaces,
that the maps are morphisms of $\holofuncs$--modules,
and that
composition in either order gives the identity.
By definition and by the surjectivity of the first map,
$\holoforms\subseteq \sheafM\subseteq \frac{1}{f}\holoforms$.
Since the second map is independent of $f$ and is the inverse of the
first, both maps are independent of $f$.

For \eqref{en:M}, define the
$\holofuncs$--submodule $\sheafD\subseteq \holovecs$ by  
$\sheafD(V)=\{\eta\in \holovecs(V):
\theta(V)(\eta,\sheafM(V))\subseteq \holofuncs(V)\}$,
and define the maps
\begin{align*}
\sheafM^*\to \sheafD\qquad\qquad &
\text{on $V$,\quad}
\varphi\mapsto \sum_{i=1}^n \varphi(V)(dx_i)\frac{\partial}{\partial x_i}
\in \sheafD(V),\\
\text{and}\quad
\sheafD\to \sheafM^*\qquad\qquad &
\text{on $V$,\quad}
\eta\mapsto 
\left(W\mapsto 
 \left(\omega\mapsto \theta(W)(\eta,\omega)\right)
\right)\in
\sheafM^*(V).
\end{align*}
Check the same conditions as for \eqref{en:D}.
By definition,
$f\cdot \holovecs\subseteq \sheafD\subseteq \holovecs$.
\end{proof}

\begin{remark}
The hypothesis that $f\cdot \holovecs\subseteq \sheafD$ for some
nonzero $f\in\holofuncs(U)$
is equivalent to
$I_n(\sheafD(U))\neq (0)$.
\end{remark}

\begin{definition}
For a $\holofuncs$--module $\sheafN$ of vector fields (respectively, $1$-forms) as in Lemma \ref{lemma:duals},
call the module constructed in \eqref{en:D} (resp., \eqref{en:M})
the \emph{realization} of $\sheafN^*$ as a module of $1$-forms (resp.,
vector fields) and denote it by $R(N)$.
\end{definition}

\begin{example}
Let $X$ be the origin in $\C^2$.
If $\sheafM=\Der_{\C^2}(-\log X)$, then
applying 
Lemma \ref{lemma:duals}\eqref{en:D}
gives $\sheafM^*\cong R(\sheafM)=\Omega^1_{\C^2}$, and
\eqref{en:M} gives $\sheafM^{**}\cong R(R(\sheafM))=\Der_{\C^2}$.
\end{example}

\begin{remark}
\label{rem:stalkwise}
For a coherent $\holofuncs$--submodule $\sheafN$ of $\holovecs$ or
$\meroforms$ as in Lemma \ref{lemma:duals},
and $p\in U$,
$R(\sheafN)_p$ depends only on $\sheafN_p$.
For, there is a canonical isomorphism
$(\sheafN^*)_p\cong \Hom_{\holofuncs_p}(\sheafN_p,\holofuncs_p)$
(\cite[A.4.4]{g-r}), and
by construction $R(\sheafN)_p$ depends only on $(\sheafN^*)_p$.
Since $\sheafN^*$ is coherent,
it is enough to understand the realization operation at
stalks. 
\end{remark}

\begin{example}
\label{ex:logarithmicforms}
For a hypersurface $X$ in $U\subseteq \C^n$ defined by
a reduced 
$f\in\holofuncs(U)$,
the sheaf of \emph{logarithmic $1$-forms}
$\Omega_U^1(\log X)$
consists of the meromorphic $1$-forms
$\omega$ such that $f\cdot \omega$ and $f\cdot d\omega$ are
holomorphic (\cite[(1.1)]{Sa}).
In \cite[(1.6)]{Sa}, Saito shows that at each $p\in U$,
$\Omega_{U,p}^1(\log X)$ and $\Derlog{X}_p$ are each the 
$\holofuncs_{U,p}$--dual of the other by the pairing $\theta_p$.
By 
coherence,
$R(\Derlog{X})=\Omega_U^1(\log X)$ and
$R(\Omega_U^1(\log X))=\Derlog{X}$.
%
\end{example}

Recall that
for a $\holofuncs$--module $\sheafN$,
there is a natural morphism $\sheafN\to \sheafN^{**}$,
and that $\sheafN$ is \emph{reflexive}
when this morphism is an isomorphism%
\footnote{For $\sheafN\subseteq \holovecs$ and
$\sheafN\subseteq \frac{1}{f} \holoforms$,
this map is automatically injective because
$\holovecs$ and $\frac{1}{f}\holoforms$ (and thus $\sheafN$)
are torsion-free.}.
If $\sheafD\subseteq \holovecs$ and Lemma \ref{lemma:duals}
is used to construct
the module $R(R(\sheafD))\subseteq \holovecs$ and
an isomorphism $\sheafD^{**}\to R(R(\sheafD))$,
then composition 
with the natural map $\sheafD\to \sheafD^{**}$ gives an
interesting 
homomorphism between two submodules of $\holovecs$.

\begin{corollary}
\label{cor:dgetsbigger}
Let $f\in\holofuncs(U)$ with $f\neq 0$,
let $\sheafD\subseteq \holovecs$ be a $\holofuncs$--submodule
containing $f\cdot \holovecs$,
and let
$i:\sheafD\to \sheafD^{**}$ be the canonical map.
If the identifications in the proof of Lemma
\ref{lemma:duals} are used to construct
an isomorphism
$j:\sheafD^{**}\to R(R(\sheafD))\subseteq \holovecs$,
then
$j\circ i$ is the inclusion map.
In particular,
$\sheafD\subseteq R(R(\sheafD))$, and $\sheafD$ is reflexive if and
only if 
$\sheafD=R(R(\sheafD))$.
\end{corollary}
\begin{proof}
By the proof of Lemma \ref{lemma:duals}\eqref{en:D}, we have 
a $\holofuncs$--module
$\sheafM=R(\sheafD)\subseteq \meroforms$ and an isomorphism
$\rho:\sheafM\to \sheafD^*$.
By the proof of Lemma \ref{lemma:duals}\eqref{en:M}, 
we have a module $R(\sheafM)\subseteq \holovecs$
and an isomorphism $\sigma:\sheafM^*\to R(\sheafM)=R(R(\sheafD))$.
Let
$\kappa:\sheafD^{**}\to \sheafM^*$ be defined by
$\kappa(V)(\varphi)=(W\mapsto \varphi(W)\circ \rho(W))\in\sheafM^*(V)$
for open $W\subseteq V\subseteq U$.
Then $j:\sheafD^{**}\to R(R(\sheafD))$ in the statement is
defined by $j=\sigma\circ \kappa$.

For open $W\subseteq V\subseteq U$ and $\eta\in \sheafD(V)$,
$(\kappa\circ i)(V)(\eta)\in
\Hom_{\holofuncs|_V}(\sheafM|_V,\holofuncs|_V)$
is defined by
$W\mapsto (\omega\mapsto \theta(W)(\eta,\omega))$.
Writing $\eta$ in coordinates and applying $\sigma(V)$
shows that
$(\sigma\circ \kappa\circ i)(V)$ is the identity on $\sheafD$, so
$j\circ i$ is the inclusion.

Since $j$ is an isomorphism and $i$ must be injective,
$\sheafD$ is reflexive if and only if
$i$ is surjective, if and only if $j\circ i$ is surjective, that is,
$\sheafD=(j\circ i)(\sheafD)=R(R(\sheafD))$.
\end{proof}

\begin{remark}
\label{rem:formsgetbigger}
An analogous result, proved similarly, holds for 
a module $\sheafM$ of forms with
$\holoforms\subseteq \sheafM\subseteq \frac{1}{f}\cdot\holoforms$.
\end{remark}

We shall need the following property of reflexive sheaves.

\begin{lemma}
\label{lemma:codim2}
For $i=1,2$, let
$\sheafD_i\subseteq \holovecs$ be a coherent
$\mathscr{O}$--module, with
some nonzero $f_i\in\holofuncs(U)$ such that $f_i\cdot \holovecs\subseteq
\sheafD_i\subseteq \holovecs$.
If $\sheafD_1$ and $\sheafD_2$ are equal off a set of codimension
$\geq 2$, then $R(\sheafD_1)=R(\sheafD_2)$ are reflexive. 
\end{lemma}
\begin{proof}
By Remark \ref{rem:stalkwise}, 
$R(\sheafD_1)$ is equal to $R(\sheafD_2)$
off the same set of codimension $\geq 2$.
Each $R(\sheafD_i)$ is reflexive and coherent,
as is the case for the $\holofuncs$--dual of any coherent sheaf
(\cite[A.4.4]{g-r}).
Since $U$ is normal, 
coherent reflexive sheaves which are equal off a set of codimension $\geq 2$
are
equal (\cite[Prop. 1.6]{hartshorne-stablereflexive}).
\end{proof}

By Remark \ref{rem:stalkwise}, 
realization is a well-defined operation on
the stalk of a coherent $\holofuncs$--submodule
$\sheafN$
satisfying the hypotheses of Lemma \ref{lemma:duals},
and so there is a clear definition for such a stalk being
\emph{reflexive}.
By coherence, Corollary \ref{cor:dgetsbigger}, and
Remark \ref{rem:formsgetbigger},
$\sheafN_p$ is reflexive if and only if
$\sheafN_p=R(R(\sheafN_p))$,
and these conditions at stalks are equivalent to
those for $\sheafN$ restricted to a small enough
open set containing $p$.
We call $R(R(\sheafN_p))$ the
\emph{reflexive hull} of $\sheafN_p$.

There is the following characterization of
reflexive modules of logarithmic vector fields.

\begin{proposition}
\label{prop:reflexiveiff}
Let $(X,p)$ be an analytic germ in $\C^n$, and
let $(H,p)$ be the union of the hypersurface components
of $(X,p)$,
setting $H=\emptyset$ if $\dim(X)\neq n-1$.
Then
\begin{enumerate}
\item
\label{en:reflexiveiff}
$\Derlog{X}_p$ is reflexive if and only if
either $(X,p)$ is empty, is $(\C^n,p)$, or is the hypersurface germ $(H,p)$.
\item
\label{en:R}
$R(\Derlog{X}_p)=R(\Derlog{H}_p)$ is the module of logarithmic $1$-forms
for $(H,p)$.
\item
\label{en:RR}
$R(R(\Derlog{X}_p))=\Derlog{H}_p$.
\end{enumerate}
\end{proposition}
\begin{proof}
For one direction of \eqref{en:reflexiveiff},
$\Derlog{\emptyset}_p=\Derlog{\C^n}_p=\Der_{\C^n,p}$ is
free and thus
reflexive.
If $(X,p)$ is a hypersurface germ,
necessarily $(H,p)$, then 
$\Derlog{X}_p$ is reflexive by \cite[(1.7)]{Sa}.

For \eqref{en:R},
choose representatives of $(X,p)$ and $(H,p)$ on an open set $U$
chosen as in Lemma \ref{lemma:goodrep}.
Let $Y'$ be the union of the codimension $\geq 2$ components of $X'$,
so that $X'=H'\cup Y'$.
Since $\Derlog{X'}$ equals $\Derlog{H'}$ off of $Y'$, by Lemma
\ref{lemma:codim2},
$R(\Derlog{X'})=R(\Derlog{H'})$.
This gives the equality of \eqref{en:R}, and the
interpretation as logarithmic forms is due to Saito
(see Example \ref{ex:logarithmicforms}).

By \eqref{en:R} we have
$R(R(\Derlog{X}_p))=R(R(\Derlog{H}_p))$.
Since $\Derlog{H}_p$ is reflexive
by \eqref{en:reflexiveiff},
this proves \eqref{en:RR}.

To finish \eqref{en:reflexiveiff}, 
if $\Derlog{X}_p$ is reflexive then
by \eqref{en:RR},
$\Derlog{X}_p=\Derlog{H}_p$.
If $\Derlog{X}_p=\Der_{\C^n,p}$, then
either $(X,p)$ is empty or is $(\C^n,p)$;
if not, then $(X,p)=(H,p)$,
which must be a hypersurface.
\end{proof}

\begin{remark}
A coherent $\holofuncs$--module $\sheafF$ is reflexive if and only if
(at least locally) it can be written in an exact sequence
$$
\xymatrix{
0\ar[r] & \sheafF \ar[r] & \sheafE \ar[r] & \sheafG \ar[r] & 0,
}$$
where $\sheafE$ is locally free and $\sheafG$ is torsion-free
(\cite[Prop 1.1]{hartshorne-stablereflexive}).
For a reduced hypersurface $X$ defined near $p$ by $f$, we have
$$
\xymatrix{
0\ar[r] & \Derlog{X} \ar[r]^-{\alpha} & \holovecs\oplus
\holofuncs\ar[r]^-\beta & (f,\mathrm{jac}(f)) \ar[r] & 0,
}$$
where $\mathrm{jac}(f)$ is the Jacobian ideal generated by the partial
derivatives of $f$,
$\alpha(\eta)=(\eta,-\frac{\eta(f)}{f})$, and
$\beta(\eta,g)=\eta(f)+f\cdot g$. 
%
%
%
%
%
\end{remark}

\section{Hypersurfaces and free divisors}
\label{sec:hypersurfaces}

We now apply our earlier results to hypersurface germs.
If $L$ is a module of vector fields logarithmic to a hypersurface
$(X,p)$, we give a criterion for $R(R(L))=\Derlog{X}_p$.
This generalizes criteria of Saito and Brion for free divisors.

\subsection{Hypersurfaces}
\label{subsec:hypersurfaces}
First, we summarize our earlier results for a hypersurface component
of an analytic germ.

\begin{proposition}
\label{prop:hypersurface}
Let $(X,p)$ be a germ of an analytic set in $\C^n$.
Let $\eta_1,\ldots,\eta_m\in\Derlog{X}_p$,
$L_{\C}=\C\{\eta_1,\ldots,\eta_m\}$,
and $L=\mathscr{O}_{\C^n,p}\cdot L_{\C}$.
Let $(X_0,p)$ be an irreducible hypersurface component of $(X,p)$.
Choose representatives of $(X,p)$, $(X_0,p)$, and each $\eta_i$, and
let $U$ be an open neighborhood of $p$ on which
these representatives satisfy the conditions of Lemma
\ref{lemma:goodrep}. 
Then the following conditions are equivalent:
\begin{enumerate}
\item
\label{cond:ideal}
$I_n(L)\nsubseteq (I(X_0))^2$ in $\mathscr{O}_{\C^n,p}$;
\item
\label{cond:pointsarbclose}
for every open neighborhood $V\subseteq U$ containing $p$,
there exists a $q\in X'_0$ satisfying one of the
following equivalent conditions:
  \begin{enumerate}
  \item \label{cond:dgnonzero} there exists $g\in I_n(L')$ such that
$dg_{(q)}\neq 0$;
  \item \label{cond:gnotinm2}  there exists $g\in I_n(L')$ such that
$g\notin \mathscr{M}_q^2$ in $\mathscr{O}_{\C^n,q}$;
  \item \label{cond:genderlog} $q$ is a smooth point of $X'$ and
$\mathscr{O}_{\C^n,q}\cdot L'=\Derlog{X'}_q$;
  \item \label{cond:geometric} $\dim(\bracket{L'}_q)=n-1$ and
there exists a nonzero $\xi\in L'_{\C}$ vanishing at $q$ such that, if
$f_0\in\mathscr{O}_{\C^n,q}$ is a reduced defining equation
for $(X'_0,q)$, then one of the following
equivalent conditions holds:
    \begin{enumerate}
    \item \label{cond:xif0}  $\xi(f_0)=\gamma\cdot f_0$, with
$\gamma(q)\neq 0$;
    \item \label{cond:dxif0} $\xi(f_0)$ has a nonzero derivative at
$q$;
    \item \label{cond:dxi}   $\im(d(\widehat{\xi})_{(q)})\nsubseteq
\bracket{L'}_q$;
    \end{enumerate}
  \end{enumerate}
\item \label{cond:everysmoothpt}
for every open neighborhood $V\subseteq U$ containing $p$, there exists
a closed analytic set $Y\subseteq V\cap X_0'$ of codimension $\geq 2$ in $V$ such that 
at every $q\in(\smooth(X)\cap X_0)\setminus Y$,
$L'$ and $(X',q)=(X'_0,q)$ satisfy the
equivalent conditions of \eqref{cond:pointsarbclose}. 
\end{enumerate}
\end{proposition}
\begin{proof}
First, observe that if any of the conditions of
\eqref{cond:pointsarbclose} hold for some $q\in X_0'$, then $q$ is a smooth
point of $X'$:
either $g$ locally defines a smooth hypersurface
(necessarily $(X',q)=(X'_0,q)$),
$q$ is assumed
smooth, or we use Lemma \ref{lemma:smoothpt}.

Since $(X_0,p)$ is a hypersurface germ, by Lemma \ref{lemma:ci}, $I(X_0)^2=(I(X_0))^{(2)}$.
Proposition \ref{prop:componentreduced} shows that
\eqref{cond:ideal}
is equivalent to
one of several equivalent conditions (listed in Proposition
\ref{prop:smoothideals2}) which should be satisfied at
some
point of $\smooth(X')\cap X_0'\cap V$ for every open neighborhood
$V\subseteq U$ containing $p$, 
or equivalently, at all points in a Zariski open subset of
$\smooth(X')\cap X_0'\cap V$
for every open neighborhood $V\subseteq U$ containing $p$.
One of these equivalent conditions, Proposition
\ref{prop:smoothideals2}\eqref{en:sm2a}, is the same as
\eqref{cond:gnotinm2}.

It remains only to check that the conditions of
\eqref{cond:pointsarbclose} are equivalent.
By a simple argument, \eqref{cond:dgnonzero} is equivalent to
\eqref{cond:gnotinm2}.
By Proposition \ref{prop:smoothideals2},
\eqref{cond:gnotinm2} is equivalent to
$\dim(\bracket{L}_q)=n-1$ and the existence of a nonzero $\xi\in L'_\C$
vanishing at $q$ with $\alpha(\xi)\neq 0$ (equivalently,
$\beta(\xi)\neq 0$).
Since $I(X'_0)=(f_0)$ in $\mathscr{O}_{\C^n,q}$
and $\xi\in\Derlog{X'_0}_q$, $\xi(f_0)=\gamma\cdot f_0$ for some
$\gamma\in \mathscr{O}_{\C^n,q}$.
By the definition of $\alpha$ and $\beta$,
\eqref{cond:xif0} is the condition that $\alpha(\xi)\neq 0$,
and \eqref{cond:dxi} is the condition the $\beta(\xi)\neq 0$.
To see that \eqref{cond:dxif0} is equivalent to \eqref{cond:xif0},
use the product rule on the equation $\xi(f_0)=\gamma\cdot f_0$
and the observation that $d(f_0)_{(q)}\neq 0$ as $(X'_0,q)$ is
smooth and $f_0$ reduced.

It remains to prove that \eqref{cond:genderlog} is equivalent to,
e.g., \eqref{cond:xif0}.
But $\alpha$ is a $\C$-linear map to a $1$-dimensional vector space, and
hence
$\alpha$ is nonzero if and only if $\alpha$ is surjective.
Thus, the equivalence follows by Theorem \ref{thm:smoothgencriteria}.
\end{proof}

\begin{remark}
Let $f\in\Ocnp$ define a reduced hypersurface $(X,p)$.
In conditions \eqref{cond:xif0} or \eqref{cond:dxif0}, 
$f_0$ could be any reduced defining equation for $(X_0,q)$,
including (a representative of) $f$.
Since $(\frac{1}{\gamma}\xi)(f)=f$ in $\mathscr{O}_{\C^n,q}$,
these conditions imply that
$\frac{1}{\gamma} \xi$ is an \emph{Euler-like vector field} for 
$f$ in some neighborhood of $q$.
This neighborhood may not include $p$, as
there are hypersurfaces without Euler-like vector fields.
\end{remark}

\begin{remark}
Let $(X,p)$ be a hypersurface.
If a free $\mathscr{O}_{\C^n,p}$--module $L\subseteq \Derlog{X}_p$ of
rank $n$ has $\dim(\bracket{L}_q)=n$ for $q\notin X$, then
$L$ is called a \emph{free* structure for $(X,p)$} in 
\cite{damon-legacy2}.
If $L\neq \Derlog{X}_p$, then
Proposition \ref{prop:hypersurface} gives 
a geometric interpretation of how $L$ must 
differ from $\Derlog{X}_p$.
\end{remark}

Using the notation and results of \S\ref{sec:duals},
we have the following criterion for a set of
logarithmic vector fields to generate all
such vector fields for a hypersurface germ. 

\begin{theorem}
\label{thm:generalizesaito}
Let $(X,p)$ be a hypersurface germ in $\C^n$ defined locally by a reduced
$f\in\Ocnp$.
Let $L$ be a submodule of
$\Derlog{X}_p$.
If $I_n(L)\subseteq (h)$ for a reduced $h$ implies that $h|f$
(equivalently,
the hypersurface component of the analytic germ
$(Z,p)$ defined by $I_n(L)$ is $(X,p)$),
and every irreducible hypersurface component $(X_0,p)$ of $(X,p)$
satisfies one of the equivalent
conditions of Proposition \ref{prop:hypersurface}, then
$R(L)$ is the module of germs of logarithmic $1$-forms of $(X,p)$
and
$$R(R(L))=\Derlog{X}_p.$$
If also $L$ is reflexive, then $L=\Derlog{X}_p$.
\end{theorem}
\begin{proof}
The equivalence of the two conditions is straightforward.
%
%
Choose representatives of $X$ and $L$, and
choose an open set $U$ containing $p$ satisfying Lemma \ref{lemma:goodrep}.
Define $Z'$ using $I_n(L')$.
Let $\sheafL'$ be the $\mathscr{O}_U$--module generated by $L'$. 

Write $Z'=X'\cup Y'$, where $Y'$ consists of the irreducible components
of $Z'$ having codimension $\geq 2$.
At $q\notin Z'$, $\sheafL'_q$, $\holovecs_q$, and $\Derlog{X'}_q$ are equal.
Let $X'_i$, $i=1,\ldots,k$, be the irreducible hypersurface
components of $X'$.
By Proposition \ref{prop:hypersurface}
there is an analytic set $B_i\subseteq X_i$ of codimension $\geq
2$ in $U$ such that
$\sheafL'$ and $\Derlog{X'}$ are equal at every
$q\in (\smooth(X')\cap X'_i)\setminus B_i$.

Thus, $\sheafL'$ and $\Derlog{X'}$ are equal off 
$\sing(X')\cup Y'\cup \bigcup_{i=1}^k B_i$, which is of codimension
$\geq 2$.
By Lemma \ref{lemma:codim2} and
Proposition \ref{prop:reflexiveiff}\eqref{en:R},
$R(L)=R(\Derlog{X}_p)$ is the module of
logarithmic $1$-forms.
By Proposition \ref{prop:reflexiveiff}\eqref{en:RR},
$R(R(L))=R(R(\Derlog{X}_p))=\Derlog{X}_p$.
For the final statement, use
the last claim of Corollary \ref{cor:dgetsbigger}.
\end{proof}

\begin{remark}
Conversely, if $L=\Derlog{X}_p$ for a hypersurface $(X,p)$, then
$L$ is reflexive and the hypotheses of Theorem
\ref{thm:generalizesaito} are satisfied by Proposition
\ref{prop:reflexiveiff}\eqref{en:reflexiveiff} and
Remark \ref{rem:reduced}.
\end{remark}

\begin{example}
Let $f\in \mathscr{O}_{\C^n,p}$ define a reduced 
hypersurface germ $(X,p)$.
Let $L\subseteq \Derlog{X}_p$ be the module generated by the
vector fields of Example \ref{ex:dumbvfs}.
At $q\notin X$,
there exist $n$ linearly independent elements of $L$.
At every $q\in\smooth(X)$,
$L$ will satisfy, e.g., 
Proposition \ref{prop:hypersurface}\eqref{cond:xif0}.
Thus by Theorem \ref{thm:generalizesaito},
$R(L)$ is the module of logarithmic $1$-forms for $X$, and
$R(R(L))=\Derlog{X}$.

That such a generic construction works may be surprising,
but the algebraic conditions for
a
$\varphi\in \Hom_{\mathscr{O}_{\C^n,p}}(f\cdot \Der_{\C^n,p},\mathscr{O}_{\C^n,p})$
to extend (uniquely) to $L$, and for
a corresponding $\omega\in \frac{1}{f} \Omega^1_{\C^n,p}$ to be
logarithmic to $(X,p)$,
are the same.
\end{example} 

\subsection{Free divisors}
\label{sec:fd}

A hypersurface germ $(X,p)$ in $\C^n$
is called a \emph{free divisor} if $\Derlog{X}_p$ is a free
module, necessarily of rank $n$.
Theorem \ref{thm:generalizesaito} implies the following
result,
for which the equivalence of
\eqref{cond:isfd} and
\eqref{cond:saitoideal}
is due to Saito.

\begin{corollary}[Saito's criterion, {\cite[(1.8)ii]{Sa}}]
\label{cor:saitoscriterion}
Let $(X,p)$ be a hypersurface germ defined locally by
a reduced
$f\in\mathscr{O}_{\C^n,p}$.
The following are equivalent:
\begin{enumerate}
\item
\label{cond:isfd}
$(X,p)$ is a free divisor;
\item
\label{cond:saitoideal}
there exists
$\eta_1,\ldots,\eta_n\in\Derlog{X}_p$,
such that
for $L=\mathscr{O}_{\C^n,p}\{\eta_1,\ldots,\eta_n\}$,
$I_n(L)=(f)$ in $\mathscr{O}_{\C^n,p}$;
\item
\label{cond:geocriteria}
there exists
$\eta_1,\ldots,\eta_n\in\Derlog{X}_p$,
linearly independent off $X$,
such that
for every irreducible component of $(X_0,p)$ of $(X,p)$,
$L=\mathscr{O}_{\C^n,p}\{\eta_1,\ldots,\eta_n\}$
satisfies one of the equivalent conditions of
Proposition \ref{prop:hypersurface}.
\end{enumerate}
Moreover,
$L=\Derlog{X}_p$.
\end{corollary}
\begin{proof}
If \eqref{cond:isfd}, then let $\eta_1,\ldots,\eta_n$ be a free basis of
$\Derlog{X}_p=L$.
Since the vector fields are linearly independent off $(X,p)$, linearly dependent on $(X,p)$,
and the principal ideal $I_n(L)=(g)$
must satisfy the sharpness conditions of Theorem \ref{thm:onminors},
we have 
$I_n(L)=(f)$, which is \eqref{cond:saitoideal}.  

If \eqref{cond:saitoideal}, then $\eta_1,\ldots,\eta_n$ are linearly independent off $(X,p)$.
Since $f$ is reduced,
for every $(X_0,p)$, $L$ satisfies
condition Proposition \ref{prop:hypersurface}\eqref{cond:ideal}.
This proves
\eqref{cond:geocriteria}.

If \eqref{cond:geocriteria}, then the free (thus reflexive) module $L$ satisfies the hypotheses of
Theorem \ref{thm:generalizesaito}, so $L=\Derlog{X}_p$.
This proves \eqref{cond:isfd}.
\end{proof}

\begin{remark}
\label{rem:ngeneratorsforfd}
For a free divisor $(X,p)$, the equivalent conditions of
Proposition \ref{prop:hypersurface}\eqref{cond:pointsarbclose}
or \eqref{cond:everysmoothpt}
are satisfied at every smooth point of $X$ (that is, $Y=\emptyset$).
This follows from the coherence of $\Derlog{X}$ and condition 
\eqref{cond:genderlog}.
\end{remark}

\begin{remark}
There is a second ``Saito's criterion''
(see \cite[(1.9)]{Sa}): if 
$L$ is the module generated by $\eta_1,\ldots,\eta_n\in
\Der_{\C^n,p}$,
$L$ is closed under the Lie bracket of vector fields,
$(X,p)$ is defined as a set by $I_n(L)=(g)$,
and $g\in\mathscr{O}_{\C^n,p}$ is reduced, 
then $(X,p)$ is a free divisor and $L=\Derlog{X}_p$.
To prove this, use a generalization of the
Frobenius Theorem (e.g., \cite{nagano})
to show that $L\subseteq \Derlog{X}_p$, and then
apply Corollary \ref{cor:saitoscriterion}.

By the same argument, there are versions of Theorem
\ref{thm:generalizesaito}
and Corollary \ref{cor:saitoscriterion}
where $L$ is a submodule of
$\Der_{\C^n,p}$
closed under the Lie
bracket of vector fields, and $(X,p)$ is the
hypersurface components of the set defined by $I_n(L)$,
\end{remark}

\subsection{Linear free divisors}
\label{subsec:lfds}
We now show that
Corollary \ref{cor:saitoscriterion} generalizes a theorem
of Brion concerning `linear' free divisors.
Here, we work in the algebraic category.

Let $V$ be a complex vector space of dimension $n$ and let $D\subseteq
V$ be a reduced hypersurface.  We say $D$ (or $(D,0)$) is a
\emph{linear free divisor} if
$\Derlog{D}_0$ has a free basis of 
$n$ vector fields which are \emph{linear}, homogeneous of degree $0$
(e.g., $(3x-2y)\partial_x-z\partial_y$).
By Saito's criterion, a linear free divisor in $V$ must be
defined by a homogeneous polynomial of degree $n$.
All linear free divisors arise from a rational representation
$\rho:G\to\GL(V)$ of a connected complex linear algebraic group $G$
with $n=\dim(G)$,
a Zariski open orbit $\Omega$, and with $D=V\setminus \Omega$
(see \cite[\S2]{gmns}).

For now, let $\rho:G\to\GL(V)$ be a rational representation of a connected
complex linear algebraic group $G$ with Lie algebra $\g$ and a Zariski open
orbit $\Omega$.
Differentiating $\rho$ gives a Lie algebra homomorphism
$d\rho_{(e)}:\g\to\End(V)$.
Since $V$ is a vector space, we can give a
canonical identification $\phi_v:V\to T_vV$ for each $v\in V$, and
then define a Lie algebra (anti-)homomorphism
$\tau:\g\to\Derlog{(V\setminus \Omega)}$ by
$\tau(X)(v)=\phi_v\left( d\rho_{(e)}(X)(v) \right)$
(see \cite{DP-matrixsingI}).
Thus 
$\tau(\g)$ is a finite-dimensional Lie algebra of 
linear vector fields, logarithmic to $V\setminus \Omega$.
For a linear free divisor $D$, there is a representation so that
$\tau(\g)$ generates the module $\Derlog{D}_0$.

Michel Brion 
used his work on log-homogeneous varieties (\cite{brionlog}) to
prove the following necessary and sufficient condition for
$D$ to be a linear free divisor.

\begin{corollary}[\cite{brion}, {\cite[Theorem 2.1]{freedivisorsinpvs}}]
\label{cor:brions}
Let $V$ be a complex vector space of dimension $n$,
and let $D\subseteq V$ be a reduced hypersurface.
Let $G\subseteq \GL(V)$ be the largest connected
subgroup which preserves $D$, with Lie algebra $\g$.
Let $\rho:G\to\GL(V)$ be the inclusion map.
Then the following are equivalent:
\begin{enumerate}
\item \label{cond:dlfd} $(D,0)$ is a linear free divisor and $\tau(\g)$
generates $\Derlog{D}_0$;
\item \label{cond:othertwo} Both: 
\begin{enumerate}
 \item \label{cond:vmdopen}     $V\setminus D$ is a unique $G$-orbit, and the corresponding
isotropy groups are finite; and
 \item \label{cond:isotropyrep} The smooth part in $D$ of each irreducible component of $D$ is
a unique $G$-orbit, and the corresponding isotropy groups are
extensions of finite groups by the multiplicative group
$\Gm=(\C\setminus\{0\},\cdot)$.
\end{enumerate}
\end{enumerate}
\end{corollary}

Our proof differs from \cite{freedivisorsinpvs} by
using Corollary \ref{cor:saitoscriterion} instead of
\cite{brionlog}.

\begin{proof}
Let $L=\mathscr{O}_{\C^n,p}\cdot \tau(\g)$.

For $v\in V$, let $G_v$ denote the isotropy subgroup at $v$, and
let $G^0_v$ be the identity component of $G_v$.
The Lie algebra $\g_v$ of $G_v$ consists 
of those $Y\in\g$ such that $\tau(Y)$ vanishes at $v$,
and $T_v(G\cdot v)=\bracket{\tau(\g)}_v$. 
Thus \eqref{cond:vmdopen} implies that $n=\dim(G)$;
as this is also true for \eqref{cond:dlfd},
assume $n=\dim(G)$.

Suppose that $v\in V$ has $\dim(G\cdot v)=n-1$,
and hence $\overline{G\cdot v}$ is a hypersurface defined by
a reduced, irreducible $f\in \C[V]$.
Then $\rho$ induces a representation $\rho_v:G_v\to\GL(N)$ on the
normal space $N$ to $G\cdot v$ at $v$, and by assumption
$\dim(N)=\dim(G_v)=1$.
It follows that $\rho_v$ acts on $N$ by multiplication by a character
$G_v\to\Gm$.
By a Lemma in \cite{pike-numcomponents},
if $\rho$ has an open orbit then this character is
the restriction of some $\chi:G\to\Gm$ with
$f(\rho(g)(w))=\chi(g)\cdot f(w)$ for all $g\in G$ and $w\in V$.
Setting $g=\exp(t\cdot X)$ and differentiating, we see that for $X\in \g_v$,
\begin{equation}
\label{eqn:dchi}
\tau(X)(f)=d\chi_{(e)}(X)\cdot f,
\end{equation}
where $d\chi_{(e)}(X)\in\C$.
By \eqref{eqn:dchi}, the constant function
$d\chi_{(e)}(X)$ plays the role of $\gamma$ in 
condition 
Proposition \ref{prop:hypersurface}\eqref{cond:xif0}.
Thus, this condition is satisfied
at such a $v$ for $L$
if and only if $\rho_v|_{G^0_v}$ is
nontrivial (and hence is an isomorphism $G^0_v\to \Gm$).

If \eqref{cond:dlfd}, then by the above observations
and Corollary \ref{cor:saitoscriterion},
\eqref{cond:vmdopen} and \eqref{cond:isotropyrep} follow readily,
except that we only know the tangent spaces to the claimed orbits.
Mather's Lemma on Lie Group Actions,
an understanding of the connectedness of the smooth locus of a complex
analytic set,
and Lemma \ref{lemma:smoothpt} may be combined to show the orbits are as claimed.

If \eqref{cond:othertwo}, then $\tau(\g)$ is a $n$-dimensional vector
space of vector
fields with 
$\dim(\bracket{\tau(\g)}_v)=n$
for all $v\notin D$,
and
$\dim(\bracket{\tau(\g)}_v)=n-1$
for all $v\in\smooth(D)$.
All that remains before applying
Corollary \ref{cor:saitoscriterion} is to show that
for $v\in\smooth(D)$, $\rho_v|_{G^0_v}$ is nontrivial.
Since
$G^0_v$ is reductive by assumption,
$\rho|_{G^0_v}$ decomposes as a direct sum of representations. 
It follows
that the normal line may be realized as an actual
$1$-dimensional subspace $W$ of $V$, complementary to
$(\phi_v)^{-1}(T_v(G\cdot v))\subseteq V$.
If $\rho_v|_{G^0_v}$ is trivial, then
$\rho|_{G^0_v}$ fixes all points in $W$ and hence
fixes all points in $v+W$; as $v+W$ is a line transverse to $D$, it intersects
$V\setminus D$.
As a result, a trivial $\rho_v|_{G^0_v}$ contradicts
\eqref{cond:vmdopen}.
\end{proof}

\bibliographystyle{amsalpha}
\bibliography{refs}

\providecommand{\bysame}{\leavevmode\hbox to3em{\hrulefill}\thinspace}
\providecommand{\MR}{\relax\ifhmode\unskip\space\fi MR }
\providecommand{\MRhref}[2]{%
  \href{http://www.ams.org/mathscinet-getitem?mr=#1}{#2}
}
\providecommand{\href}[2]{#2}
\begin{thebibliography}{GMNRS09}

\bibitem[Bri06]{brion}
Michel Brion, \emph{Some remarks on linear free divisors}, E-mail to
  Ragnar-Olaf Buchweitz, September 2006.

\bibitem[Bri07]{brionlog}
\bysame, \emph{Log homogeneous varieties}, Proceedings of the {XVI}th {L}atin
  {A}merican {A}lgebra {C}olloquium ({S}panish), Bibl. Rev. Mat.
  Iberoamericana, Rev. Mat. Iberoamericana, Madrid, 2007, pp.~1--39.
  \MR{2500349 (2010m:14063)}

\bibitem[Dam03]{damon-legacy2}
James Damon, \emph{On the legacy of free divisors. {II}. {F}ree{${}^*$}
  divisors and complete intersections}, Mosc. Math. J. \textbf{3} (2003),
  no.~2, 361--395, 742, Dedicated to Vladimir I. Arnold on the occasion of his
  65th birthday. \MR{2025265 (2005d:32048)}

\bibitem[DP12]{DP-matrixsingI}
James Damon and Brian Pike, \emph{Solvable groups, free divisors and
  nonisolated matrix singularities {I}: Towers of free divisors}, Submitted.
  \href{http://arxiv.org/abs/1201.1577}{arXiv:1201.1577 [math.AG]}, 2012.

\bibitem[DP14]{DP-matrixsingII}
James Damon and Brian Pike, \emph{Solvable groups, free divisors and
  nonisolated matrix singularities {II}: vanishing topology}, Geom. Topol.
  \textbf{18} (2014), no.~2, 911--962. \MR{3190605}

\bibitem[EH79]{eisenbud-hochster}
David Eisenbud and Melvin Hochster, \emph{A {N}ullstellensatz with nilpotents
  and {Z}ariski's main lemma on holomorphic functions}, J. Algebra \textbf{58}
  (1979), no.~1, 157--161. \MR{535850 (80g:14002)}

\bibitem[Eis95]{eisenbud}
David Eisenbud, \emph{Commutative algebra}, Graduate Texts in Mathematics, vol.
  150, Springer-Verlag, New York, 1995, With a view toward algebraic geometry.
  \MR{1322960 (97a:13001)}

\bibitem[GMNRS09]{gmns}
Michel Granger, David Mond, Alicia Nieto-Reyes, and Mathias Schulze,
  \emph{Linear free divisors and the global logarithmic comparison theorem},
  Ann. Inst. Fourier (Grenoble) \textbf{59} (2009), no.~2, 811--850.
  \MR{2521436 (2010g:32047)}

\bibitem[GMS11]{freedivisorsinpvs}
Michel Granger, David Mond, and Mathias Schulze, \emph{Free divisors in
  prehomogeneous vector spaces}, Proc. Lond. Math. Soc. (3) \textbf{102}
  (2011), no.~5, 923--950. \MR{2795728}

\bibitem[GR84]{g-r}
Hans Grauert and Reinhold Remmert, \emph{Coherent analytic sheaves},
  Grundlehren der Mathematischen Wissenschaften [Fundamental Principles of
  Mathematical Sciences], vol. 265, Springer-Verlag, Berlin, 1984. \MR{755331
  (86a:32001)}

\bibitem[Har80]{hartshorne-stablereflexive}
Robin Hartshorne, \emph{Stable reflexive sheaves}, Math. Ann. \textbf{254}
  (1980), no.~2, 121--176. \MR{597077 (82b:14011)}

\bibitem[HM93]{hausermuller}
Herwig Hauser and Gerd M{\"u}ller, \emph{Affine varieties and {L}ie algebras of
  vector fields}, Manuscripta Math. \textbf{80} (1993), no.~3, 309--337.
  \MR{1240653 (94j:17025)}

\bibitem[Nag66]{nagano}
Tadashi Nagano, \emph{Linear differential systems with singularities and an
  application to transitive {L}ie algebras}, J. Math. Soc. Japan \textbf{18}
  (1966), 398--404. \MR{0199865 (33 \#8005)}

\bibitem[Pik]{pike-numcomponents}
Brian Pike, \emph{Additive relative invariants and the components of a linear
  free divisor}, \href{http://arxiv.org/abs/1401.2976}{arXiv:1401.2976
  [math.RT]}.

\bibitem[Sai80]{Sa}
Kyoji Saito, \emph{Theory of logarithmic differential forms and logarithmic
  vector fields}, J. Fac. Sci. Univ. Tokyo Sect. IA Math. \textbf{27} (1980),
  no.~2, 265--291. \MR{586450 (83h:32023)}

\bibitem[Vas98]{vasconcelos}
Wolmer~V. Vasconcelos, \emph{Computational methods in commutative algebra and
  algebraic geometry}, Algorithms and Computation in Mathematics, vol.~2,
  Springer-Verlag, Berlin, 1998, With chapters by David Eisenbud, Daniel R.
  Grayson, J{\"u}rgen Herzog and Michael Stillman. \MR{1484973 (99c:13048)}

\end{thebibliography}

\end{document}